\documentclass[reqno]{amsart}
\usepackage{graphicx}
\usepackage{xcolor,bbm,easybmat,bm}
\usepackage{enumerate}
\usepackage[shortlabels]{enumitem}
\usepackage{amsmath,amsfonts,amssymb,hyperref,mathrsfs,esint}

\DeclareMathOperator{\loc}{loc\ }
\DeclareMathOperator{\divg}{div}
\DeclareMathOperator{\dist}{dist}
\DeclareMathOperator{\osc}{osc\ }
\DeclareMathOperator*{\fiint}{\ensuremath{\iint\text{\kern-1.36em{\raisebox{5.87pt}{\rotatebox{-93}{$\setminus$}}}}}\,}

\theoremstyle{plain}
\newtheorem{theorem}[equation]{Theorem}

\newtheorem{lem}[equation]{Lemma}

\newtheorem{thm}[equation]{Theorem}
\newtheorem{cor}[equation]{Corollary}

\theoremstyle{definition}
\newtheorem{defn}[equation]{Definition}

\theoremstyle{remark}

\newtheorem{re}[equation]{Remark}

\numberwithin{equation}{section}
\newcommand{\norm}[1]{\left\Vert#1\right\Vert}
\newcommand{\abs}[1]{\left\vert#1\right\vert}
\newcommand{\br}[1]{\left(#1\right)}
\newcommand{\bbr}[1]{\Big(#1\Big)}
\newcommand{\set}[1]{\left\{#1\right\}}
\newcommand{\Real}{\mathbb R}
\newcommand{\Rn}{\mathbb R^n}

\newcommand{\eps}{\varepsilon}

\newcommand{\A}{\mathcal{A}}
\newcommand{\LL}{\mathcal{L}}

\newcommand{\vp}{\varphi}

\newcommand{\wt}{\widetilde}
\newcommand{\bdy}{\partial}

\newcommand{\stcomp}[1]{{#1}^{\mathsf{c}}}

\newcommand{\R}{\mathbb{R}} 
\newcommand{\cN}{\mathfrak{N}} 
 
\newcommand{\cA}{\mathfrak{A}}
\newcommand{\sm}{\setminus} 
\renewcommand{\epsilon}{\varepsilon}

\renewcommand{\d}{\partial} 
\newcommand{\ms}{\medskip}

\begin{document}

\title[Carleson estimates for the Green function]
{Carleson estimates for the Green function on domains with lower dimensional boundaries}
\thanks{
G. David was partially supported by the European Community H2020 grant GHAIA 777822,
and the Simons Foundation grant 601941, GD.
S. Mayboroda was partly supported by the NSF RAISE-TAQS grant DMS-1839077 and the Simons foundation grant 563916, SM}

\author{Guy David}
\author{Linhan Li}
\author{Svitlana Mayboroda}

\newcommand{\Addresses}{{% additional braces for segregating \footnotesize
  \bigskip
  \footnotesize

 Guy David, \textsc{Universit\'e Paris-Saclay, CNRS, 
 Laboratoire de math\'ematiques d'Orsay, 91405 Orsay, France} 
 \par\nopagebreak
  \textit{E-mail address}: \texttt{guy.david@universite-paris-saclay.fr}

  \medskip

  Linhan Li, \textsc{School of Mathematics, University of Minnesota, Minneapolis, MN 55455, USA}\par\nopagebreak
  \textit{E-mail address}: \texttt{li001711@umn.edu}

  \medskip

  Svitlana Mayboroda, \textsc{School of Mathematics, University of Minnesota, Minneapolis, MN 55455, USA}\par\nopagebreak
  \textit{E-mail address}: \texttt{svitlana@math.umn.edu}

}}

\begin{abstract}

In the present paper we
consider an elliptic divergence form operator in $\Rn\setminus\Real^d$ with $d<n-1$ and 
prove that its Green function is almost affine, 
in the sense that the
normalized difference between the Green function with a sufficiently far away pole and 
a suitable affine function at every scale satisfies a Carleson measure estimate. 
The coefficients of the operator can be very oscillatory, and only need to satisfy some condition similar to the traditional quadratic Carleson condition.

\end{abstract}

\maketitle

\tableofcontents

\section{Introduction and main results}
In a recent paper \cite{DLMgreen}, we showed that for a slightly larger class of elliptic operators than the Dahlberg-Kenig-Pipher operators on the upper half-space $\Real^{d+1}_+$, the Green function is well approximated by affine functions. The current paper extends this result to higher co-dimensions. 
That is, we consider the 
Green function on $\Real^{n}\setminus\Real^d$, with $d<n-1$, for operators satisfying a condition analogous to the Dahlberg-Kenig-Pipher condition on $\Real^{n}\setminus\Real^d$ and show that it is close, in a suitable sense, to affine functions. There are multiple challenges specific to the higher-codimensional setting, but before discussing those, let us provide some context for this work. 

There has been a wide success in establishing connections between the geometry of the boundary of $\Omega\subset \R^n$ and properties of solutions of an elliptic PDE on $\Omega$ 
(\cite{HM2014uniform},\cite{HMU2014uniform},\cite{HMMTZ2020uniform},\cite{AHMMT2019harmonic}, etc). However, when the boundary of $\Omega$ has dimension lower than $n-1$, results are relatively rare. Essentially the only characterization of the uniform rectifiability of a lower-dimensional set by a PDE property is the recent work \cite{david2020approximation}. However, it pertains to weak rather than strong estimates on the solutions and, in particular, yields qualitative rather than quantitative results. This not merely a technical obstacle: the proofs in \cite{david2020approximation}, relying on the blow-up techniques, are not amenable to a more quantitative analysis. On the other hand, the free boundary results obtained in \cite{david2020approximation} are even stronger than perhaps is natural to expect. Specifically, the authors show that even weak estimates on the Green function imply uniform rectifiability, and hence, if one can show that the Green function is close to the distance to the boundary in a strong, quantifiable sense, this would furnish the first quantifiable 
PDE characterization of the lower-dimensional uniform rectifiability. The present paper is the first step in this direction. 

Aside from the aforementioned weak results, it has two important pre-runners. In \cite{DLMgreen}, we managed to prove that the Green function is close to the distance function in a precise, quantitative way in the upper half-space (that is, in co-dimension 1). In \cite{DFM2021Green}, a different in form but similar in spirit, quantitative estimate for the Green function is obtained on domains with uniformly rectifiable sets of dimension strictly less than $n-1$ using a completely different method. The goal of this paper is to obtain a precise estimate for the Green function for more general operators than the ones considered in \cite{david2020approximation} and \cite{DFM2021Green} on domains with lower dimensional boundary. Roughly speaking, the operators considered in \cite{david2020approximation} and \cite{DFM2021Green} are close to the analogues of the Laplacian. In the present paper, we consider operators with much more oscillatory coefficients, albeit trading off by considering only flat boundary. Let us be more precise.

Consider $\Omega=\Rn\setminus\Gamma$, where $\Gamma\subset\Rn$ is Ahlfors-regular of dimension $d<n-1$. This means that there is a constant $C_0\ge1$ such that 
\begin{equation}\label{AR}
    C_0^{-1}r^d\le\mathcal{H}^d(\Gamma\cap B_r(x))\le C_0 r^d,
\end{equation}
for all balls $B_r(x)$ centered on $x\in\Gamma$, with radius $r>0$. Classical elliptic operators are not appropriate for boundary value problems on $\Omega$, as their solutions cannot ``see" the lower dimensional set $\Gamma$.  To overcome this obstacle, the first and third authors of the present paper, together with J. Feneuil , developed an elliptic theory on such domains with degenerate elliptic operators (\cite{david2017elliptic}). It was shown that the general results,
such as the maximum principle, trace and extension theorems, existence of the harmonic
measure and Green function, all hold for the operators
\[\LL = -\divg (\A \dist(\cdot, \Gamma)^{d+1-n}\nabla),\]
where $\dist(\cdot, \Gamma)$ is the Euclidean distance to the boundary, and $\A$ is a matrix of real, bounded, measurable functions that satisfies the usual ellipticity conditions. 
That is, there is some $\mu_0>1$ such that 
\begin{equation}\label{cond ellp}
    \begin{aligned}
\langle \A(X)\xi,\zeta\rangle \le \mu_0\abs{\xi}\abs{\zeta} 
&\quad\text{for  }X \in \Omega \text{ and  }\xi, \eta\in\Rn,
\cr 
\quad \langle \A(X)\xi,\xi\rangle \ge \mu_0^{-1}\abs{\xi}^2
&\quad\text{for  }X \in \Omega \text{ and  }\xi\in\Rn.
\end{aligned}
\end{equation}
Some of the results in this general setting are included in Section \ref{sec pre}.

For the purpose of this paper, we focus only on $\Gamma=\set{(x,t)\in\Rn: t=0}\cong\Real^d$, and our domain is $\Omega=\Rn\setminus\Real^d=\set{(x,t)\in\Real^d\times\Real^{n-d}: t\neq 0}$. Notice that in this case, for a point $X=(x,t)\in\Rn$, $\dist(X,\Gamma)=\abs{t}$.

Before introducing our conditions on the operator, let us define Carleson measures on the upper half-space $\Real^{d+1}_+$. We shall systematically use lower case letters for points in $\Real^d$ and capital letters for points in $\Rn$. It will be necessary to distinguish a ball in $\Rn$ from a ball in $\Real^{d+1}$, so we use the cumbersome notation $B_r^{(d+1)}(x)$ for a ball in $\Real^{d+1}$ with radius $r$ centered at $(x,0)\in\Real^{d+1}$. The main purpose of defining balls in $\Real^{d+1}$ is to define Carleson balls in $\Real^{d+1}_+$, that is, we let
$T(x,r)=B_r^{(d+1)}(x)\cap\Real^{d+1}_+$. Although we do not emphasize it in notation, $T(x,r)$ is $(d+1)$-dimensional. For $x\in \R^d$ and $r > 0$, we denote by $\Delta(x,r)$ the surface ball $B_r(x) \cap \Gamma$.  Thus $\Delta(x,r)$ is a ball in $\R^d$, and $T(x,r)$ is a half ball in $\Real^{d+1}_+$ over $\Delta(x,r)$. 
We may simply write $T_\Delta$ for a half ball over $\Delta\subset\Real^d$.
\begin{defn}[Carleson measures on $\Real^{d+1}_+$]\label{d13}
We say that a nonnegative Borel measure $\mu$ is a Carleson measure on $\Real_+^{d+1}$, 
if its Carleson norm 
\[\norm{\mu}_{\mathcal{C}}:=\underset{\Delta\subset \Real^d}{\sup}\frac{\mu(T_\Delta)}{\abs{\Delta}}\]
is finite, where 
the supremum is over all the surface balls $\Delta$ and 
$\abs{\Delta}$ is the Lebesgue measure of $\Delta$ in $\Real^d$. We use $\mathcal{C}$ to denote the set of Carleson measures on $\Real^{d+1}_+$. 

For any 
surface ball
$\Delta_0\subset\Real^d$, we use $\mathcal{C}(\Delta_0)$ to denote the set of Borel measures satisfying the Carleson condition restricted to $\Delta_0$, 
i.e., such that
\[\norm{\mu}_{\mathcal{C}(\Delta_0)}:=\underset{\Delta\subset \Delta_0}{\sup}\frac{\mu(T_\Delta)}{\abs{\Delta}} < +\infty
.\]
\end{defn}

Next we want to define our conditions that say that the matrix $\A = \A(X)$ is often close to a ``constant" coefficient matrix $\A_0$. But since our operators have a singular weight $\abs{t}^{d+1-n}$, we need to impose some structural assumptions on the matrix $\A_0$ so that the operator $\LL_0:=-\divg(\mathcal{A}_0\abs{t}^{d+1-n}\nabla)$ behaves like a constant coefficient operator in $\Rn\setminus\Real^d$. 

It was observed in \cite{david2019new} that given an elliptic operator $L=-\divg(\wt A\nabla)$ defined on $\Real^{d+1}_+$, one can construct a degenerate elliptic operator $\LL=-\divg(\A\nabla)$
so that if $v$ is a solution to $Lv=0$ in $\Real^{d+1}_+$, then the function $u$ defined on $\Rn\setminus\Real^d$ by $u(x,t)=v(x,\abs{t})$ is a solution to $\LL u=0$ on $\Rn\setminus\Real^d$. The precise construction is the following. Consider a $(d+1)\times (d+1)$ matrix $\wt A$ written in a block form as 
\[
  \wt A =\begin{bmatrix}
\begin{BMAT}{c.c}{c.c}
A & \mathbf{b}  \\
\mathbf{c} & d
\end{BMAT}
\end{bmatrix},
\]
where $A$ is a $d\times d$ matrix, $\mathbf{b}$ is a $d\times 1$ vector, $\mathbf{c}$ is a $1\times d$ vector, and $d$ is a scalar function. Then for $n>d+1$, the $n\times n$ matrix $\A$ is constructed from $\wt A$ as 
\begin{equation}\label{A fmd+1}
    \A =\begin{bmatrix}
\begin{BMAT}{c.c}{c.c}
A & \frac{\mathbf{b}\, t}{\abs{t}} \\
\frac{t^T\mathbf{c}}{\abs{t}}& dI_{n-d}
\end{BMAT}
\end{bmatrix},
\end{equation}
where $I_{n-d}$ is the identity matrix of size $n-d$,  $t$ is seen as a horizontal vector in $\Real^{n-d}$, and thus $\mathbf{b}\, t$ is a $d\times (n-d)$ matrix and $t^T\mathbf{c}$ is a $(n-d)\times d$ matrix. 

Inspired by this observation, we fix the aforementioned class of matrices constructed from constant matrices in $\Real^{d+1}$.
\begin{defn}[The class $\cA_0(\mu_0)$] \label{d15} 
We define $\cA_0(\mu_0)$ to be the class of $n\times n$ matrices satisfying the ellipticity conditions \eqref{cond ellp} with constant $\mu_0$ that can be written as the following block matrix 
\begin{equation}\label{eq A0 def}
  \A_0 =\A_0(x,t)=\begin{bmatrix}
\begin{BMAT}{c.c}{c.c}
A_0 & \frac{\mathbf{b_0} t}{\abs{t}} \\
\frac{t^T\mathbf{c_0}}{\abs{t}}& d_0I_{n-d}
\end{BMAT}
\end{bmatrix}.
\end{equation}
Here, $A_0$ is a $d\times d$ constant matrix, $\mathbf{b_0}$ is a $d\times 1$ constant vector, $\mathbf{c_0}$ is a $1\times d$ constant vector, $d_0$ is a real number.
\end{defn}

The reason that this class of matrices plays the role of constant matrices for our purpose is actually different from the above observation made in \cite{david2019new}. We want them to relate back to constant-coefficient operators in $\Real^{d+1}$, not the other way around. In fact, it is shown in Section~\ref{sec cnnct} that for any $\A_0\in\cA_0(\mu_0)$, any solution of $-\divg(\A_0\nabla u)=0$ can be transformed into a solution of an elliptic equation in $\Real^{d+1}$. Notice that a solution $u(x,t)$ of $-\divg(\A_0\nabla u)=0$ is not necessarily radial in $t$, while a solution constructed from a solution of an elliptic equation in $\Real^{d+1}_+$ as above is radial in $t$.

Now let us return to conditions on $\A$. Since we shall compare $\A$ and $\A_0\in\cA_0(\mu_0)$ at every scale, we introduce Whitney regions 
in $\Rn$: for any $(x,r)\in\Real^{d+1}_+$, define 
\begin{equation}\label{def Wxr}
    W(x,r):=\set{(y,t)\in\Rn: y\in\Delta(x,r),\,\frac{r}{2}\le\abs{t}\le r}.
\end{equation}
Notice that $W(x,r)$ is an annular region in $\Rn$ whose distance to $\Gamma$ is $r/2$.

The difference between $\A$ and some matrix $\A_0\in\cA_0(\mu_0)$ at a given scale is measured by the following quantity. For $x\in\Real^d$ and $r>0$, define
\begin{equation}\label{def alpha}
    \alpha(x,r):=\inf_{\A_0 \in \cA_0(\mu_0)}
\bigg\{\frac{1}{m(W(x,r))}\int_{(y,t) \in W(x,r)} |\A(y,t)-\A_0|^2\,\frac{dydt}{\abs{t}^{n-d-1}}\bigg\}^{1/2}
\end{equation} 
Here, $m(W(x,r))$ is the measure of $W(x,r)$ with weight 
${\abs{t}^{-n+d+1}}$. 

\begin{defn}[Weak DKP condition]\label{def wDKP}
We say that the coefficient matrix $\A$ satisfies the weak DKP condition with
constant $M > 0$, if $\alpha(x,r)^2 \frac{dxdr}{r}$ is a 
Carleson measure on $\R^{d+1}_+$, with norm
\begin{equation}\label{eq wDKP}
\cN(\A) : = \norm{\alpha(x,r)^2 \frac{dxdr}{r}}_{\mathcal{C}} \le M.
\end{equation}
\end{defn}
The name comes from Dalhberg, Kenig and Pipher. In 1984, Dahlberg first  conjectured  that a Carleson condition on the coefficients, which is roughly that 
$\abs{\nabla A}^2dxdr/r$ be 
a Carleson measure on $\Real^{d+1}_+$, guarantees the 
absolute continuity of the elliptic measure with respect 
to the Lebesgue measure. 
In 2001, Kenig and Pipher \cite{kenig2001dirichlet} proved Dahlberg's conjecture. 

The condition we consider here is weaker than the classical DKP condition in the following sense. 
Consider a matrix $\A$ of bounded, measurable functions defined on $\Rn$ that can be written as \eqref{A fmd+1}, but with the coefficients depending on $x,t$. Assume that $A$, $\mathbf{b}$, $\mathbf{c}$ and $d$ all satisfy the usual DKP condition with Carleson norm $M$. That is, 
\begin{equation*}
    \norm{\sup_{(y,t)\in W(x,r)}\abs{\nabla A(y,t)}^2rdxdr}_{\mathcal{C}}\le M,
\end{equation*}
and 
similarly for $\mathbf{b}$, $\mathbf{c}$ and $d$. One can verify that under this assumption, the matrix $\A$ satisfies the weak DKP condition with constant $M$. 
We point out that from our definition, a matrix $\A$ that satisfies the
weak DKP condition does not have to be of the form \eqref{A fmd+1}. Moreover, we can always add to $\A$  a matrix $\mathcal{D}$ that satisfies 
\[
d\mu(x,r)=\sup_{(y,t)\in W(x,r)}\mathcal{D}(y,t)^2\frac{dxdr}{r}\in\mathcal{C},\]
and the new matrix still satisfies the weak DKP condition if $\A$ does. We remark that our Definition \ref{def wDKP} is the higher co-dimensional analogue of what we defined in \cite{DLMgreen}, where we say that a $(d+1)\times(d+1)$ matrix satisfies the weak DKP condition with constant $M$, if \eqref{eq wDKP} holds with $\A_0$ replaced by some constant $(d+1)\times(d+1)$ matrix 
in the definition \eqref{def alpha} of $\alpha(x,r)$.

Let us now turn to the approximation of the Green function by affine functions in higher co-dimension. In \cite{DLMgreen}, we showed that any solution in $T(x_0,R)$ that vanishes on $\Delta(x_0,R)$ is locally well approximated by affine functions in $T(x_0,R/2)$, with essentially uniform Carleson bounds. More precisely, we proved the following result.
\begin{theorem}[\cite{DLMgreen} Theorem 1.13]\label{mt d=n-1}
Let $\wt A$ be a $(d+1)\times (d+1)$ matrix of real-valued functions on
$\Real^{d+1}_+$ satisfying the ellipticity conditions with constant $\mu_0$. If $\wt A$ satisfies the weak DKP condition with some constant $M\in(0,\infty)$,
and if we are given $x_0 \in \R^d$, $R>0$, and a positive solution $u$ of 
$Lu=-\divg\br{\wt A\nabla u}=0$ in $T(x_0,R)$, with $u=0$ on $\Delta(x_0,R)$, 
then for some $C$ depending only on $d$ and $\mu_0$, there holds
\begin{equation*}
\norm{\beta_u(x,r) \frac{dxdr}{r}}_{\mathcal{C}(\Delta(x_0,R/2))} \leq C+CM,
\end{equation*}
where \[
\beta_u(x,r)=\frac{\fint_{T(x,r)}\abs{\nabla\br{u(y,t)-\lambda_{x,r}(u)\, t}}^2dydt}{\fint_{T(x,r)} |\nabla u(y,t)|^2dydt},
\]
and $\lambda_{x,r}(u) = \fint_{T(x,r)}\partial_t u(z,t)dzdt$.
\end{theorem}

In higher co-dimension, we want to measure in a similar way the closeness between a solution and an affine function in $\Rn\setminus\Real^d$. Given a positive solution $u$ of $\LL u=0$ in a ball $B_r(x)$ centered on $\Gamma$, the best affine function that approximates $u$ in $B_r(x)$ should be $\lambda_{x,r}(u)\abs{t}$, where 
\begin{equation}\label{def lambda}
     \lambda_{x,r}(u)=\frac{1}{m(B_r(x))}\int_{B_r(x)}\frac{\nabla_t u(z,t)\cdot t}{\abs{t}}\frac{dzdt}{\abs{t}^{n-d-1}}.
\end{equation}
In Section~\ref{sec cnnct}, we will see that this $\lambda_{x,r}(u)$ is indeed the best coefficient of $\abs{t}$ to approximate $u$ in $B_r(x)$, and that it is closely related to the best coefficient in the co-dimension one setting.

As in the co-dimension one case, the  proximity  of  the  two  functions is measured by the weighted $L^2$ average of the difference of the gradients divided by the weighted local energy of $u$. That is, we set
\begin{equation}\label{def J}
    J_u(x,r):=\frac{1}{m(B_r(x))}\int_{B_r(x)}\abs{\nabla_{y,t}\br{u(y,t)-\lambda_{x,r}(u)\abs{t}}^2}\frac{dydt}{\abs{t}^{n-d-1}},
\end{equation}
and then divide by
\begin{equation}\label{def E}
    E_u(x,r):=\frac{1}{m(B_r(x))}\int_{B_r(x)}\abs{\nabla u(y,t)}^2\frac{dydt}{\abs{t}^{n-d-1}},
\end{equation}
to get the number
\begin{equation}\label{def beta}
    \beta_u(x,r):=\frac{J_u(x,r)}{E_u(x,r)}.
\end{equation}

The solutions 
considered here 
are all weak solutions in a weighted Sobolev space. Their values on the boundary $\Gamma=\Real^d$ are considered in the trace sense. All 
this is made precise in Section~\ref{sec pre}, and also in Section~\ref{subsec fs}.
Our main result is the following.
\begin{thm}\label{mt}
Let $\A$ be an $n\times n$ matrix of bounded, real-valued functions on 
$\Real^n$ satisfying the ellipticity conditions \eqref{cond ellp}. 
If $\A$ satisfies the weak DKP condition with some constant $M\in(0,\infty)$, and if we are given $x_0 \in \R^d$, $R>0$, and a nonnegative solution $u\in W_r(B_R(x_0))$ of 
$\LL u=-\divg_{x,t}\br{\A(x,t) \abs{t}^{d+1-n}\nabla_{x,t} u}=0$ in $B_R(x_0)\setminus\Gamma$, with $Tu=0$  on $\Gamma\cap B_R(x_0)$, 
then the function $\beta_u$ defined by \eqref{def beta} satisfies a Carleson condition in $T(x_0,R/2)$, and more precisely
\begin{equation}\label{eq mt1}
\norm{\beta_u(x,r) \frac{dxdr}{r}}_{\mathcal{C}(\Delta(x_0,R/2))} \leq C+CM
\end{equation}
where $C$ depend 
only on $d$, $n$ and  $\mu_0$.
\end{thm}

The next theorem is an improvement of Theorem \ref{mt}, which says that we can have Carleson norms for $\beta_u$
that are as small as we want, provided that we take a small DKP constant and a suitably large ball where $u$ is a positive solution that vanishes on the boundary.

\begin{theorem}\label{mt2}
Let $x_0\in\Real^d$, $R>0$, $\mu_0>0$ be given, let $u$ satisfy the assumptions of Theorem \ref{mt}, and let $\A$ satisfy the weak DKP condition in $\Delta(x_0,R)$.
Then for $\tau \leq 1/2$  
\begin{equation}\label{1a15} 
\norm{\beta_u(x,r) \frac{dxdr}{r}}_{\mathcal{C}(\Delta(x_0, \tau R))} 
\leq C \tau^a + C \norm{\alpha_2(x,r)^2 \frac{dxdr}{r}}_{\mathcal{C}(\Delta(x_0,R))},
\end{equation}
where $C$ and $a > 0$ depends only on $d$, $n$ and $\mu_0$.
\end{theorem}

Finally, let us comment that our results are essentially optimal. In \cite{DLMgreen}, we constructed an example that shows that $\beta_{G_L^\infty}(x,r)\frac{dxdr}{r}$ may not be a Carleson measure if an operator $L=-\divg(\wt A\nabla)$  does not satisfy the DKP condition. Here, $G_L^\infty$ is the Green function with pole at infinity for $L$ on the upper half-plane $\Real^{2}_+$. Construct an operator $\LL=-\divg(\A\nabla)$ from the $2$- dimensional operator $L$ as in \eqref{A fmd+1}. One can show that this operator does not satify the DKP condition either. Moreover, the corresponding Green function is  $G_{\LL}^\infty(x,t)=G_{L}^\infty(x,\abs{t})$, and a similar computation as in the co-dimensional one setting shows that $\beta_{G_{\LL}^\infty}(x,r)\frac{dxdr}{r}$ cannot be a Carleson measure on $\Real^{2}_+$. 

The main differences in the proof, compared to the setting of co-dimension $1$,
lie in the decay estimates for the non-affine part of solutions to equations with a coefficient matrix in the class $\cA_0(\mu_0)$. In the co-dimension one case, we have good estimates for the second derivatives of solutions to equations with constant coefficients. This enables us to control the oscillations of the gradient of solutions. However, in the higher co-dimensional setting, the coefficients have a singular weight $\abs{t}^{-n+d+1}$, which prevents us from getting an estimate for the second derivatives of solutions. To overcome this difficulty, we split the solution into one part which is radial in $t$, and the other part which is purely rotational in $t$. The radial part can be treated similarly to the co-dimension one case, while the rotational part requires a compactness argument and other properties of solutions. The entire Section~\ref{sec Ju0} is devoted to implementing this idea. The decay estimate is proved in the key lemma (Lemma \ref{lem Ju0 decay}).

The rest of the paper is organized as follow. In Section~\ref{sec pre}, we collect some results that will be used frequently in the rest of the paper; most of them are proved in \cite{david2017elliptic}. In Section~\ref{sec cnnct}, we relate the $n$-dimensional operator $\LL_0$ back to a $d+1$- dimensional operator $L$, and transform solutions of $\LL_0u=0$ into solution of $Lv=0$. Also, we study the properties of $\lambda_{x,r}$ in that section.
In Section~\ref{sec L}, we show how to generalize the decay estimates from operators with a coefficient matrix in $\cA_0(\mu_0)$ to weak DKP operators. The ideas in that section are similar to those in the co-dimensional one case, and we mainly illustrate the modifications needed in the higher co-dimension. We give a proof of the reverse H\"older inequalities for 
the gradient of solutions, where we have to address the issue of mixed-dimensional boundaries.

\section{Preliminaries}\label{sec pre}

In this section we recall, mostly from \cite{david2017elliptic}, how to extend standard results for elliptic
PDE's in the upper half space (or NTA domains) to the setting of co-dimension $>1$. The familiar reader can probably jump to Section \ref{sec cnnct} and return to this section when needed.

Consider $\Omega=\Rn\setminus\Gamma$, where $\Gamma\subset\Rn$ is Ahlfors-regular 
of dimension $d<n-1$. In all the 
other sections, $\Gamma$ will be simply $\Real^d$. 
For $X\in\Omega$, 
write $\delta(X):= \dist(X,\Gamma)$. Define the 
weight function $w(X):=\delta(X)^{-n+d+1}$, and a measure $dm(X)=w(X)dX$. 
Denote by $B_r(X)$ the open 
ball in $\Rn$ centered at $X$ with radius $r$. One can show that 
\begin{align}
    m(B_r(X))\approx r^nw(X) \qquad\text{if } \delta(X)\ge 2r,\label{mBr far}\\
     m(B_r(X))\approx r^{d+1} \qquad\text{if } \delta(X)\le 2r.\label{mBr near}
\end{align}
In particular, this implies that that $m$ is a doubling measure. See \cite{david2017elliptic},  Chapter 2 for details.

Denote by $W = \dot{W}^{1,2}_w(\Omega)$ the weighted Sobolev
space of functions $f\in L^1_{loc}(\Omega)$ whose distributional gradient in $\Omega$ lies in $L^2(\Omega,w)$:
\begin{equation}\label{def W}
    W:= \set{f\in L^1_{loc}(\Omega): \nabla f \in L^2
(\Omega, w)}=\set{f\in L^1_{loc}(\Rn): \nabla f \in L^2
(\Rn, w)},
\end{equation} 
and set $\norm{f}_W =\br{\int_\Omega\abs{\nabla f(X)}^2w(X)dX}^{1/2}$ for $f\in W$. Here, the identity 
(i.e., the fact that the distribution derivative of $f$ on $\Omega$ can also be used as a derivative on $\R^n$)
is proved in \cite{david2017elliptic}, Lemma 3.2. We shall also use the following local version of the space $W$. Let $O\subset\Rn$ be an open set, then
\begin{equation}\label{def Wr}
    W_r(O):=\set{f\in L^1_{loc}(O): \vp f\in W \text{ for any }\vp\in C_0^\infty(O)}.
\end{equation}
Note that \(W_r(O)=\set{f\in L^1_{loc}(O): \nabla f\in L^2_{loc}(O,w)}\); % , 
see \cite{david2017elliptic} Chapter 8 for details.

For functions in $W$ or $W_r(O)$, it is shown in \cite{david2017elliptic} that there exists a well-defined trace on $\Gamma$, or $\Gamma\cap O$, respectively. The trace of $u\in W$ is such that for $\mathcal{H}^d$-almost every $x\in \Gamma$,
\begin{equation}\label{eq trace0}
    Tu(x)=\lim_{r\to0}\fint_{B(x,r)}u(Y)dY:=\lim_{r\to0}\frac{1}{\abs{B(x,r)}}\int_{B(x,r)}u(Y)dY.
\end{equation}
For $u\in W_r(O)$, the trace is defined in the same way for $\mathcal{H}^d$-almost every $x\in \Gamma\cap O$.

Consider the divergence-form operator $\LL=-\divg_X(\A(X)w(X)\nabla_X)$,
where $\A$ is  
an $n\times n$ matrix of real, bounded, measurable functions defined in $\Omega$, 
that satisfies 
the ellipticity conditions \eqref{cond ellp}. 

\begin{defn}\label{def wsol}
We say that $u\in W$ is a (weak) solution of $\LL u=0$ in $\Omega$ if for any $\vp\in C_0^\infty(\Omega)$,
\[
\int_{\Omega}\A\nabla u\cdot\nabla\vp\, dm=0.
\]
Let $O\subset\Rn$ be an open set. We say that $u\in W_r(O)$ is a (weak) solution of $\LL u=0$ in $O$ if for any $\vp\in C_0^\infty(O)$,
\(
\int_{O}\A\nabla u\cdot\nabla\vp dm=0
\).
We say that $u\in W_r(O)$ is a subsolution (respectively, supersolution) in $O$ if for any $\vp\in C_0^\infty(O)$ such that $\vp\ge 0$, \(
\int_{O}\A\nabla u\cdot\nabla\vp dm\le 0\) (respectively,\, $\ge 0$).

\end{defn}

We collect some basic properties for functions in $W$ and solutions of $\LL u=0$ in this section. The constant $C$ below might be different from line to line, but depends only on $d$, $n$, the Ahlfors constant $C_0$, and the ellipticity constant $\mu_0$ unless otherwise stated.

\begin{lem}[Poincar\'e Inequality (\cite{david2017elliptic}, Lemma 4.2)]\label{lem Poincare} Let $p\in [1,\frac{2n}{n-2}]$ (or $p\in [1, +\infty)$ if $n = 2$). Then for any $u\in W$, any ball $B\subset\Rn$ with radius $r > 0$, there is some constant $C$ depending only on $n$, $d$ and $C_0$, such that
\[
\br{\frac{1}{m(B)}\int_{B}\abs{u-u_{B}}^pdm}^{1/p}\le Cr\br{\frac{1}{m(B)}\int_{B}\abs{\nabla u}^2dm}^{1/2}
\]
where $u_B$ denotes either $\fint_Bu$ or $m(B)^{-1}\int_B udm$.
If $B$ is centered on $\Gamma$ and if, in addition, $Tu=0$  on $\Gamma\cap B$, then
\[
\br{\frac{1}{m(B)}\int_B\abs{u}^pdm}^{1/p}\le C r\br{\frac{1}{m(B)}\abs{\nabla u}^2dm}^{1/2}.
\]
\end{lem}

\begin{re}
One also has (see the proof of Lemma 4.2 in \cite{david2017elliptic}) 
\begin{equation}\label{eq Poincare u-uB 2-2+}
    \br{\frac{1}{m(B)}\int_{B}\abs{u-u_{B}}^2dm}^{1/2}\le Cr\br{\frac{1}{m(B)}\int_{B}\abs{\nabla u}^{\frac{2n}{n+2}}dm}^{\frac{n+2}{2n}}.
\end{equation}
Moreover, if $B$ is centered on $\Gamma$ and if, in addition, $Tu=0$  on $\Gamma\cap B$, then
\begin{equation}\label{eq Poincare u=0 2-2+}
     \br{\frac{1}{m(B)}\int_{B}\abs{u}^2dm}^{1/2}\le Cr\br{\frac{1}{m(B)}\int_{B}\abs{\nabla u}^{\frac{2n}{n+2}}dm}^{\frac{n+2}{2n}}.
\end{equation}
To see \eqref{eq Poincare u=0 2-2+}, write 
\[
\br{\frac{1}{m(B)}\int_{B}\abs{u}^2dm}^{1/2}\le C\br{\frac{1}{m(B)}\int_B\abs{u-\fint_Bu}^2dm}^{1/2}+C\fint_B\abs{u(X)}dX.
\]
By Lemma 4.1 of \cite{david2017elliptic} and H\"older's inequality,
\[
\fint_B\abs{u(X)}dX\le Cr\br{\frac{1}{r^{d+1}}\int_B\abs{\nabla u}^{\frac{2n}{n+2}}dm}^{\frac{n+2}{2n}}.
\]
Note that since $B$ is centered on $\Gamma$, $m(B)\approx r^{d+1}$. Thus, \eqref{eq Poincare u=0 2-2+} follows from the above observation and \eqref{eq Poincare u-uB 2-2+}.
\end{re}

\begin{lem}[Interior Caccioppoli inequality (\cite{david2017elliptic}, Lemma 8.6)]
Let $B$ be a ball of radius $r$ such that $2B\subset\Omega$ and $u\in W_r(2B)$ 
is a nonnegative subsolution of $\LL$
 in $2B$. 
Then there exists a constant $C>0$ depending only on $d,n$, $C_0$ and $\mu_0$, such that for any constant $c\in\Real$, 
\[
\int_B\abs{\nabla u}^2dm\le C r^{-2}\int_{\frac{3}{2}B}\abs{u-c}^2dm.
\]
\end{lem}

\begin{lem}[Caccioppoli inequality on the boundary (\cite{david2017elliptic} Lemma 8.11)]\label{lem bdyCaccio}
Let $B\subset\Rn$ be a ball of radius $r$ centered on $\Gamma$, and let $u\in W_r(2B)$ be a nonnegative subsolution in $2B\setminus\Gamma$ such that $Tu=0$  on $2B\cap\Gamma$. Then 
\[
\int_B\abs{\nabla u}^2dm\le C r^{-2}\int_{\frac3{2}B}u^2dm.
\]
\end{lem}

\begin{lem}[Moser estimates on the boundary (\cite{david2017elliptic} Lemma 8.12)]\label{lem bdyMoser} Let $B$ and $u$ be as in Lemma \ref{lem bdyCaccio}. Then 
\[
\sup_B u\le C\br{m\br{B}^{-1}\int_{\frac3{2}B}u^2dm}^{1/2}.
\]
Here, the constant $C$ depends only on $d$, $n$ and $\mu_0$ as usual.
\end{lem}

Let $B$ be a ball centered on $\Gamma$ with radius $r$. 
We say that a point $X_B$ is a {\it corkscrew point} for 
$B$ if $X_B\in B$ and $\delta(X_B)\ge\epsilon r$ for some $\epsilon$ depending only on $d$, $n$ and the Ahlfors constant $C_0$ of $\Gamma$. 

\begin{lem}[Boundary Harnack's Inequality (\cite{david2017elliptic}, Lemma 11.8)]
Let $x_0\in\Gamma$ and $r > 0$ be given, and let $X_r$ be a corkscrew point for 
$B_r(x_0)$. Let $u\in W_r(B_{2r}(x_0))$ be a nonnegative solution of $\LL u=0$ in $B_{2r}(x_0)\setminus\Gamma$, such that $Tu=0$  on $B_{2r}(x_0)\cap\Gamma$. Then
\[u(X) \le Cu(X_r) \qquad \text{for } X \in B_r(x_0).\]
\end{lem}

\begin{lem}\label{lem corkscrew}
Let $x_0\in\Gamma$ and $R>0$ be given. Suppose $u\in W_r(B_R(x_0))$ is a 
nonnegative solution of $\LL u=0$ in $B_{R}(x_0)\setminus\Gamma$ with $Tu=0$  on $B_R(x_0)\cap\Gamma$. Then for all $0<r<R/2$,
\[
\frac{u^2(X_r)}{r^2}\approx\frac{1}{m(B_r(x_0))}\int_{B_{r}(x_0)}\abs{\nabla u}^2dm,
\]
where $X_r$ is a corkscrew point of $B_r(x_0)$.
\end{lem}

\begin{proof}
By translation invariance, we may assume that the origin is on $\Gamma$ and that $x_0$ is the origin. 
To see the less than or equal to direction, we say that 
$Tu=0$ on the boundary and use Lemma \ref{lem bdyMoser} followed by Sobolev's inequality to get \[
\frac{u^2(X_r)}{r^2}\le \frac{Cr^{-2}}{m(B_r)}\int_{B_r}u^2dm\le \frac{C}{m(B_r)}\int_{B_r}\abs{\nabla u}^2dm.
\]
To see the other direction, we use the boundary Caccioppoli and boundary Harnack inequalities to get
\[
\frac{1}{m(B_{2r})}\int_{B_r}\abs{\nabla u}^2dm\le \frac{C}{r^2}\frac{1}{m(B_{2r})}\int_{B_{2r}}\abs{u}^2dm\le\frac{C}{r^2}u^2(X_r).
\]

\end{proof}

\begin{lem}[Comparison Principle (\cite{david2017elliptic}, Theorem 11.17)]\label{lem cmpsn}
Let $x_0\in\Gamma$ and $r > 0$, and let $X_r$ be a corkscrew point. Let $u, v\in W_r(B_{2r}(x_0))$ be two
nonnegative, not identically zero, solutions of $\LL u = \LL v = 0$ in $B_{2r}(x_0)\setminus\Gamma$, such that $Tu=Tv=0$  on $\Gamma\cap B_{2r}(x_0)$. Then
\[
C^{-1}\frac{u(X_r)}{v(X_r)}\le \frac{u(X)}{v(X)}\le C\frac{u(X_r)}{v(X_r)} \qquad\text{for all }\, X\in B_{r}(x_0)\setminus\Gamma,
\]
where $C > 1$ depends only on $n, d$, $C_0$ and $\mu_0$.
\end{lem}

\begin{cor}[\cite{david3square}, Corollary 6.4]\label{cor cmpsn}
Let $u$, $v$, $r$, $x_0$ be as in Lemma \ref{lem cmpsn}. There exists $C>0$ and $\gamma\in(0,1)$ depending only on $n, d$, $C_0$ and $\mu_0$, such that 
\[
\abs{\frac{u(X)v(Y)}{u(Y)v(X)}-1}\le C\br{\frac{\rho}{r}}^\gamma
\]
for all $X$, $Y\in B_\rho(x_0)\setminus\Gamma$, as long as $\rho<r/4$.
\end{cor}

We have the following reverse H\"older inequality for the gradient of solutions.
\begin{lem}\label{lem RH}
Let $B\subset\Rn$ be a ball centered on $\Gamma$. Let $u\in W_r(3B)$ be a solution of $\LL u=0$ in $3B\setminus\Gamma$ with $Tu=0$  on $3B\cap\Gamma$. Then there exist $p>2$ depending only on $d,n$, $C_0$ and $\mu_0$, and $C>0$ depending on $d,n$, $C_0$, $\mu_0$ and $p$, such that 
\begin{equation}\label{eq RH largerball}
    \br{\frac{1}{m(B)}\int_B\abs{\nabla u}^pdm}^{1/p}\le C\br{\frac{1}{m(2B)}\int_{2B}\abs{\nabla u}^2dm}^{1/2}.
\end{equation}
If in addition, $u\ge0$ in $3B$, then 
\begin{equation}\label{eq RH sameball}
   \br{\frac{1}{m(B)}\int_B\abs{\nabla u}^pdm}^{1/p}\le C\br{\frac{1}{m(B)}\int_{B}\abs{\nabla u}^2dm}^{1/2}. 
\end{equation}
\end{lem}

To prove Lemma \ref{lem RH}, we first derive the following inequality
\begin{lem}\label{lem RH 2-2+}
Let $X\in\Rn$ and $r>0$ be given.  Let $u\in W_r(B_{4r}(X))$ be a solution of $\LL u=0$ in $B_{4r}(X)\setminus\Gamma$, with $Tu=0$ on $B_{4r}(X)\cap\Gamma$ if $B_{4r}(X)\cap\Gamma$ is not empty. Then 
\begin{equation}\label{eq RH 2-2+}
    \br{\frac{1}{m(B_r(X))}\int_{B_r(X)}\abs{\nabla u}^2dm}^{1/2}\le C\br{\frac{1}{m(B_{3r}(X))}\int_{B_{3r}(X)}\abs{\nabla u}^{\frac{2n}{n+2}}dm}^{\frac{n+2}{2n}}.
\end{equation}
\end{lem}
\begin{proof}
Case 1: $\delta(X)\le\frac{5}{4}r$. Then there exists $x_0\in\Gamma$ so that $B_r(X)\subset B_{\frac{9}{4}r}(x_0)$. Hence, by Caccioppoli's inequality on the boundary and \eqref{eq Poincare u=0 2-2+}, 
\begin{align*}
    \br{\frac{1}{m(B_r(X))}\int_{B_r(X)}\abs{\nabla u}^2dm}^{1/2}&\lesssim \br{\frac{1}{m(B_{9r/4}(x_0))}\int_{B_{9r/4}(x_0)}\abs{\nabla u}^2dm}^{1/2}\\
    &\lesssim \br{\frac{1}{m(B_{5r/2}(x_0))}\int_{B_{5r/2}(x_0)}\abs{\nabla u}^{\frac{2n}{n+2}}dm}^{\frac{n+2}{2n}}.
\end{align*}
Then \eqref{eq RH 2-2+} follows from the fact that  $B_{5r/2}(x_0)\subset B_{3r}(X)$.

Case 2: $\delta(X)>\frac{5}{4}r$. Then $B_{5r/4}(X)\subset\Rn\setminus\Gamma$. By the 
interior Caccioppoli inequality and the 
Poincar\'e inequality \eqref{eq Poincare u-uB 2-2+},
\begin{align*}
    \bbr{\frac{1}{m(B_r(X))}\int_{B_r(X)}\abs{\nabla u}^2dm}^{\frac12}
    &\lesssim \frac{1}{r}\bbr{\frac{1}{m(B_{\frac{5r}{4}}(X))}\int_{B_{\frac{5r}{4}}(X)}\abs{u-u_{B_{5r/4}(X)}}^2dm}^{\frac12}\\
    &\lesssim\bbr{\frac{1}{m(B_{\frac{5r}{4}}(X))}\int_{B_{\frac{5r}{4}}(X)}\abs{\nabla u}^{\frac{2n}{n+2}}dm}^{\frac{n+2}{2n}}.
\end{align*}
\end{proof}

\begin{proof}[Sketch of proof of Lemma \ref{lem RH}]
One can deduce 
\eqref{eq RH largerball} in Lemma \ref{lem RH} from Lemma \ref{lem RH 2-2+} 
and a modification of the argument in \cite{giaquinta1983multiple} (Theorem 1.2, Chapter V). Thanks to the fact that $m$ is a doubling measure, the argument in \cite{giaquinta1983multiple} carries over. The only modification is that one should choose parameters everywhere in the argument in \cite{giaquinta1983multiple} according to the doubling constant of $m$ instead of that of Lebesgue measure in $\Rn$. Once we obtain \eqref{eq RH largerball} and assume additionally $u$ is an nonnegative solution, \eqref{eq RH sameball} follows immediately from Lemma \ref{lem corkscrew} and Harnack's inequality.
\end{proof}

\section{Connection with the co-dimensional one case: an analogue of constant-coefficient operators}\label{sec cnnct}

 From now on, we focus only on $\Omega=\Rn\setminus\Gamma$ with $\Gamma=\set{(x,t)\in\Rn: t=0}\cong\Real^d$. Notice that in this setting, for a point $(x,t)\in\Rn$, its distance to $\Gamma$ 
 is simply $\abs{t}$. Therefore, we can simply define the weight function $w$ as a function in $\Real^{n-d}$. 
 That is, for $t\in\Real^{n-d}$, define
\[
w(t):=\abs{t}^{-n+d+1}.
\]
Recall that $B_r(X)$ denotes the 
ball in $\Rn$ with radius $r$ centered at $X\in\Rn$. For $x\in\Real^d$, we write $B_r(x):=B_r(x,0)$, the 
ball in $\Rn$ with radius $r$ centered at $(x,0)\in\Rn$. 
Recall also that for a set $E$ in $\Rn$, $m(E)=\int_E w(t)\, dxdt$. As the following computation shows, for a ball in $\Rn$ centered on $\Gamma$, its $m$ measure is equal to the Lebesgue measure of a Carleson ball in $\Real^{d+1}$ multiplied by the surface area of the unit $(n-d-1)$-dimensional sphere:
\begin{align}
    m(B_r(x_0))&=\int_{B_r(x_0)}w(t)\, dxdt =
  \int_{B_r(x_0)} \abs{t}^{-n+d+1} \, dxdt    \nonumber\\
    &=\int_{\abs{x-x_0}\le r}\int_{\rho=0}^{\sqrt{r^2-\abs{x-x_0}^2}}
    \int_{\omega \in S^{n-d-1}}d\omega d\rho dx
    \nonumber\\
    &= \abs{T(x_0,r)}\sigma(S^{n-d-1})=c_{n,d}r^{d+1}.\label{eq mBandT}
\end{align}

Let $\LL=-\divg_{x,t}(\mathcal{A}(x,t)w(t)\nabla_{x,t})$ be an operator defined in $\Rn\setminus\Gamma$, where $\A(x,t)=\begin{bmatrix}
a_{ij}(x,t)
\end{bmatrix}$ is an $n\times n$ matrix of real-valued, measurable functions on $\Rn$, which satisfies the ellipticity conditions \eqref{cond ellp}. We shall systematically use $\A_0$ to denote an $n\times n$ matrix in the class $\cA_0(\mu_0)$,  and write  $\LL_0=-\divg_{x,t}(\mathcal{A}_0w(t)\nabla_{x,t})$.

The main benefit of taking $\A_0$ in this particular form is that the solutions of $\LL_0 u=0$ can be converted to solutions of a constant-coefficient equation in $\Real^{d+1}$. Let us introduce the $(d+1)$- dimensional constant-coefficient elliptic operator
\begin{equation}\label{def L0}
    L_0:=-\divg_{x,\rho}(\wt A\nabla_{x,\rho}), \qquad\text{with }
    \wt A=\begin{bmatrix}
\begin{BMAT}{c.c}{c.c}
A_0 & \mathbf{b_0} \\
\mathbf{c_0} & d_0
\end{BMAT}
\end{bmatrix},
\end{equation}
where $A_0$, $\mathbf{b_0}$, $\mathbf{c_0}$ and $d_0$ are the same as in \eqref{eq A0 def}.  Alternatively, we can write 
\begin{equation}\label{def' L0}
    L_0=-\divg_x (A_0\nabla_x)-\divg_x(\mathbf{b_0}\partial_\rho )- \partial_\rho(\mathbf{c_0}\nabla_x)-d_0\partial^2_\rho.
    \end{equation}
To relate solutions of $\LL_0 u=0$ to those of $L_0 v=0$, let us first give some definitions.

\begin{defn}\label{def utheta}
Let $f=f(x,t)$ be a function defined on $\Rn$. Write $t=\rho\,\omega$ in polar coordinates, with $\rho\in\Real_+$ and $\omega\in S^{n-d-1}$. We still denote the function in polar coordinates as $f$, that is, $f(x,t)$ and $f(x,\rho\,\omega)$ are the same function in different coordinates. For any $(x,\rho)\in \Real^{d+1}_+$, define
\begin{equation}\label{eqdef utheta}
    f_\theta(x,\rho):=\fint_{S^{n-d-1}}f(x,\rho\,\omega)d\omega.
\end{equation}
For any $(x,t)\in\Rn$, define
\begin{equation}\label{eqdef wtutheta}
    \wt{f}_\theta(x,t):=f_\theta(x,\abs{t})=\fint_{S^{n-d-1}}f(x,\abs{t}\omega)d\omega.
\end{equation}
In particular, $\wt{f}_\theta$ is a function of $n$ variables and is radial in $t$.
\end{defn}

\begin{lem}\label{lem spcs utheta} With the definitions above, the following statements hold:
\begin{enumerate}
    \item If $u\in W$, then $u_\theta\in L_{loc}^1(\Real^{d+1}_+)$, $\nabla u_\theta\in L^2(\Real^{d+1}_+)$, and $\wt{u}_\theta\in W$. 
    \item Let $x_0\in \Gamma$ and $r>0$. If $u\in W_r(B_r(x_0))$, then $$u_\theta\in W^{1,2}_{loc}(T(x_0,r))=\set{f\in L_{loc}^2(T(x_0,r)):\nabla f\in L_{\loc}^2(T(x_0,r))},$$
    and $\wt{u}_\theta\in W_r(B_r(x_0))$.
\end{enumerate}

\end{lem}

\begin{proof}
(1) We first show that $u_\theta\in L_{loc}^1(\Real^{d+1}_+)$. Let $K$ be a compact set in $\Real^{d+1}_+$. Then we can find $x_0\in\Real^d$, $r>0$ and $\epsilon>0$ so that $K\subset \set{(x,\rho)\in T(x_0,r): \rho\ge \epsilon}$. By translation invariance, we can assume that $x_0$ is the origin. Then we have 
\begin{multline*}
    \int_K\abs{u_\theta(x,\rho)}d\rho dx\le \epsilon^{-n+d+1}\int_{\abs{x}\le r}\int_\epsilon^{\sqrt{r^2-\abs{x}^2}}\abs{u_\theta(x,\rho)}\rho^{n-d-1}d\rho dx\\
    \le C_{\epsilon,n,d}\int_{\abs{x}\le r}\int_\epsilon^{\sqrt{r^2-\abs{x}^2}}\int_{S^{n-d-1}}\abs{u(x,\rho\omega)}\rho^{n-d-1}d\omega d\rho dx\\
    = C_{\epsilon,n,d}\int_{B_r}\abs{u(x,t)}dxdt<\infty,
\end{multline*}
where we have used $u\in L^1_{loc}(\Rn)$ to get the finiteness of the last term. This shows $u_\theta\in L_{loc}^1(\Real^{d+1}_+)$. 

Now we compute the $L^2$ integral of $\abs{\nabla u_\theta}$ over a Carleson ball $T_r$ centered at the origin. Observe that by the definition of $u_\theta$ and $\wt{u}_\theta$, expressing the gradient in polar coordinates, we have 
\(
\abs{\nabla_{x,\rho}u_\theta(x,\rho)}=\abs{\nabla_{x,t}\wt u_\theta(x,t)}\), for  $\rho=\abs{t}$.
Hence,
\begin{equation}\label{eq dr1}
     \abs{\nabla_{x,\rho}u_\theta(x,\rho)}^2=\abs{\nabla_{x,t}\fint_{S^{n-d-1}}u(x,\abs{t}\omega)d\omega}^2\le \fint_{S^{n-d-1}}\abs{\nabla_{x,t}u(x,\abs{t}\omega)}^2d\omega.
\end{equation}
Let $s=\abs{t}\omega$, then $\abs{s}=\abs{t}$, and \(\frac{\partial s_k}{\partial t_j}=\frac{t_j}{\abs{t}}\omega_k\), for \(k,j=1,2,\dots, n-d\). Thus,
\begin{equation*}
  \abs{\nabla_{t}u(x,\abs{t}\omega)}^2=\sum_{j=1}^{n-d}\br{\partial_{t_j}u(x,s)}^2=\sum_{j=1}^{n-d}\br{\sum_{k=1}^{n-d}\partial_{s_k}u(x,s)\frac{t_j}{\abs{t}}\omega_k}^2
    \le\abs{\nabla_su(x,s)}^2.  
\end{equation*}
Combining this with \eqref{eq dr1}, we obtain
\[
\abs{\nabla_{x,\rho}u_\theta(x,\rho)}^2\le\frac{1}{\sigma(S^{n-d-1})}\int_{\rho S^{n-d-1}}\abs{\nabla_{x,s}u(x,s)}^2\rho^{-n+d+1}ds.
\]
Therefore, integrating in polar coordinates, we can control the $L^2$ integral of $\abs{\nabla u_\theta}$ as follows.
\begin{multline}\label{eq W1utheta}
    \int_{T_r}\abs{\nabla_{x,\rho}u_\theta(x,\rho)}^2dxd\rho
    \le C_{n,d}\int_{\abs{x}\le r}\int_{\rho=0}^{\sqrt{r^2-\abs{x}^2}}\int_{\rho S^{n-d-1}}\abs{\nabla_{x,s}u(x,s)}^2\frac{dsd\rho dx}{\rho^{n-d-1}}\\
    =C_{n,d}\int_{\abs{x}\le r}\int_{\abs{t}\le\sqrt{r^2-\abs{x}^2}}\abs{\nabla_{x,t}u(x,t)}^2\frac{dt dx}{\abs{t}^{n-d-1}}\\
    = C_{n,d}\int_{B_r}\abs{\nabla_{x,t} u(x,t)}^2w(t)\, dxdt\le C_{n,d}\norm{\nabla u}_{L^2(\Rn,w)}^2.
\end{multline}
Letting $r$ go to infinity we obtain $\nabla u_\theta\in L^2(\Real^{d+1})$ with 
\[\norm{\nabla u_\theta}_{L^2(\Real^{d+1}_+)}\le C_{n,d}\norm{\nabla u}_{L^2(\Rn,w)}.\]
As for $\wt{u}_\theta$, let us fix any $r>0$ and evaluate the integral of $\wt u_\theta$ over the ball $B_r$. 
\begin{multline*}
    \int_{B_r}\abs{\wt u_\theta(x,t)}^2dxdt=\int_{\abs{x}\le r}\int_0^{\sqrt{r^2-\abs{x}^2}}\int_{S^{n-d-1}}\abs{u_\theta(x,\rho)}\rho^{n-d-1}d\omega d\rho dx\\
    \le \int_{\abs{x}\le r}\int_0^{\sqrt{r^2-\abs{x}^2}}\int_{S^{n-d-1}}\abs{u(x,\rho\omega')}d\omega'\rho^{n-d-1} d\rho dx=\int_{B_r}\abs{u(x,t)}dxdt.
\end{multline*}
This shows that $\wt{u}_\theta\in L^1_{loc}(\Rn)$. Finally, we compute 
\begin{multline*}
    \int_{B_r}\abs{\nabla_{x,t}\wt u_\theta(x,t)}^2w(t)\, dxdt=\int_{\abs{x}\le r}\int_0^{\sqrt{r^2-\abs{x}^2}}\int_{S^{n-d-1}}\abs{\nabla_{x,\rho}u_\theta(x,\rho)}^2d\omega d\rho dx\\
    =\sigma(S^{n-d-1})\int_{T_r}\abs{\nabla_{x,\rho}u_\theta(x,\rho)}^2d\rho dx.
\end{multline*}
This and \eqref{eq W1utheta} give $\nabla \wt u_\theta\in L^2(\Rn,w)$. Thus, $\wt u_\theta\in W$.

(2) By translation and dilation invariance, we can assume that $x_0$ is the origin and $r=1$. By a similar argument as in (1), one sees that if $u\in W_r(B_1)$, then $u_\theta\in L_{loc}^1(T_1)$, 
$\nabla u_\theta\in L_{loc}^2(T_1)$, and $\wt u_\theta\in W_r(B_1)$. 
So it remains to show $u_\theta\in L^2_{loc}(T_1)$. 
Notice however that $u_\theta$ lives in the upper half space, where there is no disturbing weight $w$;
then we can apply the usual Poincar\'e estimate, in the homogeneous space of locally integrable functions
$f$ such that $\nabla f\in L^2$, and indeed get that $u_\theta\in L^2_{loc}(T_1)$.
This gives 
$u_\theta\in W_{loc}^{1,2}(T_1)$.
\end{proof}

Now we can show how solutions of equations in $\Real^{d+1}_+$ and $\Rn\setminus\Real^d$ are related.

\begin{lem}\label{lem utheta sol of L0}
Let $B$ be a ball centered on $\Gamma$. If $u\in W_r(B)$ is a solution of $\LL_0u =0$ in $B\setminus\Gamma$, where $\A_0$ is in block form \eqref{eq A0 def}, then $u_\theta\in W_{loc}^{1,2}(T)$ is a solution of $L_0u_\theta=0$ in $T$, and $\wt{u}_{\theta}\in W_r(B)$ is a solution of $\LL_0\wt{u}_{\theta}=0$ in $B\setminus\Gamma$.
\end{lem}

\begin{proof}
Let us assume that $u$ is a $C^2$ function and thus a strong solution. Writing out the derivatives, we see that
\begin{multline*}
    \LL_0=\frac{-1}{\abs{t}^{n-d-1}}\br{\divg_x(A_0\nabla_x)+\divg_x\Big(\frac{\mathbf{b_0}t\cdot\nabla_ t}{\abs{t}}\Big)+\divg_t\bbr{\frac{t^T\mathbf{c_0}\nabla_x}{\abs{t}}}+\divg_t(d_0\nabla_t)}\\
    +\frac{n-d-1}{\abs{t}^{n-d}}\mathbf{c_0}\nabla_x+\frac{n-d-1}{\abs{t}^{n-d+1}}d_0t\cdot\nabla_t \, . 
\end{multline*} 
Fortunately, some of these terms cancel. In fact, 
\[
\divg_t\br{\frac{t^T\mathbf{c_0}\nabla_x}{\abs{t}}}=\frac{(n-d-1)\mathbf{c_0}\nabla_x}{\abs{t}}+\frac{t\cdot\nabla_t(\mathbf{c_0}\nabla_x)}{\abs{t}},
\]
and thus 
\begin{multline*}
    \LL_0=-\frac{1}{\abs{t}^{n-d-1}}\br{\divg_x(A_0\nabla_x)+\divg_x\br{\frac{\mathbf{b_0}t\cdot\nabla t}{\abs{t}}}
    +\frac{t\cdot\nabla_t(\mathbf{c_0}\nabla_x)}{\abs{t}}+d_0\Delta_t}\\
    +\frac{n-d-1}{\abs{t}^{n-d+1}}d_0t\cdot\nabla_t \, . 
\end{multline*}
Changing to polar coordinates $t=\rho\,\omega$, we have $t\cdot\nabla_t=\rho\partial_\rho$, and  \[\Delta_t=\partial_{\rho}^2+(n-d-1)\frac{1}{\rho}\partial_{\rho}+\frac{1}{\rho^2}\Delta_{\omega},\]
where $\Delta_{\omega}$ is the Laplacian on the sphere $S^{n-d-1}$. Then $\LL_0$ can be simplified as  
    \begin{equation}\label{eq L0 polar}
        \LL_0=-\frac1{{\rho}^{n-d-1}}\br{\divg_x (A_0\nabla_x)+\divg_x(\mathbf{b_0}\partial_\rho )+ \partial_\rho(\mathbf{c_0}\nabla_x)+d_0\partial^2_\rho}-\frac{d_0}{{\rho}^{n-d+1}}\Delta_{\omega}.
    \end{equation}
Notice that the quantity in the first parenthesis 
is exactly what we have for $L_0$ in \eqref{def' L0}.
Now since $\wt{u}_\theta$ is radial in $t$, $\Delta_\omega\wt{u}_\theta=0$, and thus,
\begin{equation}\label{eq LL0=L0}
    \LL_0\wt{u}_\theta=\frac{1}{\rho^{n-d-1}}L_0u_\theta.
   \end{equation}
   Notice that  $\int_{S^{n-d-1}}\Delta_{\omega}u(x,\rho\,\omega)d\omega=0$ by the 
   divergence theorem, so that we can add this term for free and get the following.
   \begin{equation*}
       \LL_0\wt u_\theta=-\frac1{{\rho}^{n-d-1}}L_0 u_{\theta}+\frac{d_0}{\rho^{n-d+1}}\fint_{S^{n-d-1}}\Delta_\omega u(x,\rho\,\omega)d\omega.
   \end{equation*}
   Exchanging the order of integration and differentiation, we obtain
    \begin{multline*}
    \LL_0\wt{u}_\theta=-\frac1{{\rho}^{n-d-1}}\fint_{S^{n-d-1}}\br{L_0+\frac{d_0}{\rho^2}\Delta_\omega}u(x,\rho\,\omega)d\omega
    =\fint_{S^{n-d-1}}\LL_0u\, d\omega=0.
   \end{multline*}

This and \eqref{eq LL0=L0} show that $\LL_0\wt{u}_{\theta}=0=L_0u_\theta$. 

The smoothness assumption on solutions is harmless. First, we have checked in Lemma \ref{lem spcs utheta} that given $u\in W_r(B)$, $u_\theta$ and $\wt{u}_\theta$ are in the right spaces stated in the lemma. Now if $u\in W_r(B)$ is a weak solution, it is a strong solution in any compact set in $B\setminus\Gamma$. This is because on these sets, $\abs{t}\ge\delta$ for some $\delta>0$, and thus the coefficients are smooth. Then we use our results for strong solutions and conclude that $u_\theta$ and $\wt u_\theta$ are strong solutions in any compact set in $T$ and $B\setminus\Gamma$, respectively. Then they are of course weak solutions in these compact sets. But this is all we need as in the weak formulation of equations, the test functions are compactly supported in $T$ (for $u_\theta$) and in $B\setminus\Gamma$ (for $\wt u_\theta$).
\end{proof}

\begin{re}\label{re |t| sol}
Writing $\LL_0$ in polar coordinates as in 
\eqref{eq L0 polar}, one immediately sees $\LL_0\abs{t}=0$ in $\Rn\setminus\Real^d$. We shall use this property of $\abs{t}$ in the future.
\end{re}

We now turn to the quantity $\lambda_{x,r}(u)$. First we show that $\lambda_{x,r}(u)\abs{t}$ is the best approximation of a given function $u$ in $B_r(x)$ by a multiple of $\abs{t}$.
\begin{lem}[Orthogonality]\label{lem orth}
For any $(x,r)\in\Real^{d+1}_+$, for any function $u(x,t)$, 
\begin{equation}\label{eq ortho}
  \frac{1}{m(B_r(x))}\int_{B_r(x)}\nabla (u(y,t)-\lambda_{x,r}(u) \abs{t})\cdot\nabla\abs{t}\,w(t)dydt=0.  
\end{equation}
Moreover, 
\begin{equation}\label{eq min lambda}
     \inf_{\lambda\in\Real}\frac{1}{m(B_r(x))}\int_{B_r(x)}\abs{\nabla (u(y,t)-\lambda \abs{t})}^2w(t)dydt=J_u(x,r),
\end{equation}
where $J_u(x,r)$ is defined in \eqref{def J}.
\end{lem}

\begin{proof}
For any $\lambda\in\Real$, we compute
\begin{align}\label{eq lambdacpt1}
   \nabla (u-\lambda\abs{t})\cdot\nabla\abs{t}
   =\sum_{i=1}^{n-d}\br{\partial_{t_i}u-\frac{\lambda t_i}{\abs{t}}}\frac{ t_i}{\abs{t}}=\frac{\nabla_t u\cdot t}{\abs{t}}-\lambda.
\end{align}
By the definition of $\lambda_{x,r}(u)$ in \eqref{def lambda}, $\frac{\nabla_t u\cdot t}{\abs{t}}-\lambda_{x,r}(u)$ is orthogonal to constants in $L^2(B_r(x),w)$. Therefore, using \eqref{eq lambdacpt1} with $\lambda=\lambda_{x,r}(u)$, one sees that \eqref{eq ortho} holds. Turning to \eqref{eq min lambda}, we see that for any $\lambda\in\Real$,
\begin{multline*}
    \frac{1}{m(B_r(x))}\int_{B_r(x)}\abs{\nabla (u(y,t)-\lambda \abs{t})}^2w(t)dydt\\
    = \frac{1}{m(B_r(x))}\int_{B_r(x)}\abs{\nabla (u(y,t)-\lambda_{x,r}(u) \abs{t})}^2w(t)dydt\\ + \frac{\abs{\lambda_{x,r}(u)-\lambda}^2}{m(B_r(x))}\int_{B_r(x)}\abs{\nabla \abs{t}}^2w(t)dydyt\\
    =J_u(x,r)+\abs{\lambda_{x,r}(u)-\lambda}^2\ge J_u(x,r),
\end{multline*}
where in the first equality we have used \eqref{eq ortho}.
\end{proof}
 It follows from \eqref{eq min lambda} that $J_u(x,r)\le E_u(x,r)$, which implies 
\begin{equation}\label{eq beta<1}
    \beta_u(x,r)\le 1\qquad \text{for any } (x,r)\in \Real^{d+1}_+.
\end{equation}
The following lemma shows that the best approximation of $u$ by a multiple of $\abs{t}$ in $B_r(x)$ is the same as the best approximation of $u_\theta$ in $T(x,r)$.

\begin{lem}\label{lem lambda(u) in polar}
Let $x\in\Real^d$, $r>0$, and 
$u$ be as in Lemma \ref{lem orth}. Define $u_\theta$ as in Definition \ref{def utheta}. Then
\begin{equation*}
    \lambda_{x,r}(u)=\fint_{T(x,r)}\partial_\rho u_\theta(y,\rho)dyd\rho.
\end{equation*}
\end{lem}

\begin{proof}
Without loss of generality, we may assume that $x$ is the origin and $r=1$. Passing to polar coordinates $t=\rho\omega$, and noticing that 
\(\frac{t\cdot\nabla_tu}{\abs{t}}=\frac{\rho\partial_\rho u}{\rho}\), 
we have
\begin{multline*}
    \frac{1}{m(B)}\int_{B}\frac{\nabla_tu\cdot t}{\abs{t}}w(t)\, dxdt =\frac{1}{m(B)}\int_{\abs{x}\le 1}\int_{\abs{t}\le\sqrt{1-\abs{x}^2}} \frac{\nabla_tu(x,t)\cdot t}{\abs{t}}w(t)\, dxdt\\
=\frac{1}{m(B)}\int_{\abs{x}\le1}\int_0^{\sqrt{1-\abs{x}^2}}\int_{S^{n-d-1}}\partial_\rho u(x,\rho\omega)d\omega d\rho dx.
\end{multline*}
Exchanging the order of integration and differentiation,
\begin{multline*}
\frac{1}{m(B)}\int_{B}\frac{\nabla_tu\cdot t}{\abs{t}}w(t)\, dxdt=\frac{1}{m(B)}\int_{T_1}\br{\partial_\rho\int_{S^{n-d-1}}u(x,\rho\omega)d\omega}d\rho dx\\
=\fint_{T_1}\br{\partial_\rho\fint_{S^{n-d-1}}u(x,\rho\omega)d\omega}d\rho dx=\fint_{T_1}\partial_\rho u_{\theta}(x,\rho)dxd\rho,
\end{multline*}
because $|T_1| = m(B)\sigma(S^{n-d-1})$ and as desired.
\end{proof}

\section{Estimates for solutions of $\LL_0u=0$}\label{sec Ju0}
\subsection{More about function spaces}\label{subsec fs}
When proving estimates for (weak) solutions, it is useful to allow test functions that
lie in a bigger space than $C_0^\infty$. For this reason, we now define some new spaces.

\begin{defn}
Let $O\subset\Rn$ be an open, bounded set. Define 
\begin{equation}\label{def WO}
    W(O):=\set{u\in L^1_{loc}(O): \nabla u\in L^2(O,w)}.
\end{equation}
Here $L^1_{loc}(O)$ is for the Lebesgue measure, which is more natural if we want to 
see $u$ as a distribution and talk about its gradient.
Equip $W(O)$ with the seminorm $\norm{f}_{W(O)}=\br{\int_{O}\abs{\nabla f}^2dm}^{1/2}$. Define $W_0(O)$ to be the closure of $C_0^\infty(O)$ under $\norm{\cdot}_{W(O)}$. 
\end{defn}
As we shall see, $W(O)$ plays the same role as the usual Sobolev space $W^{1,2}(O)$, and $W_r(O)$ should be compared with $W^{1,2}_{loc}(O)$.

For the purposes of this paper, we are only interested in the simple case when $O$ is a ball $B$ centered on $\Gamma$, or $O = B \sm \Gamma$.
\begin{lem}\label{lem WB}
Let $B\subset\Rn$ be a ball centered on $\Gamma$. Then 
\begin{enumerate}
    \item $W(B\setminus\Gamma)=W(B)=\set{u\in L^1(B): \nabla u\in L^2(B,w)}$;
    \item $W(B)\subset W^{1,2}(B) = \{u\in L^2(B, dX): \nabla u\in L^2(B,w)$;
    \item If $u\in W_r(2B)$, then $u\in W(B)$.
\end{enumerate}
\end{lem}
\begin{proof}
(1) Let $u \in W(B\setminus\Gamma)$ be given. By definition, 
$u \in L^1_{loc}(B \sm \Gamma, dX) =  L^1_{loc}(B \sm \Gamma, w dX)$, and 
by  Lemma 3.2 in \cite{david2017elliptic}, 
$u\in L^{1}_{loc}(B, dX)$.  So $u \in W(B)$; 
we still need to check that $u \in L^1(B,dX)$. However Poincar\'e's inequality (Lemma \ref{lem Poincare}, with $p=1$) says that $u \in L^1(B, wdX)$, and then an easy estimate ((2.13) in \cite{david2017elliptic}) 
shows that $u \in L^1(B,dX)$. Notice that although our assumptions, for instance in Lemma \ref{lem Poincare},
appear to be global, we never use the values of $u$ outside of $B$.

(2) For $u\in W(B)$, we now apply Poincar\'e's inequality (Lemma \ref{lem Poincare}), now with $p=2$,
to find that 
\[ \int_{B}\abs{u-u_B}^2 wdX\le C \int_B\abs{\nabla u}^2 dm,\]
Then again $u \in L^2(B,dX)$ by (2.13) in \cite{david2017elliptic},
this time applied to $g = \abs{u-u_B}^2$.

(3) follows immediately from (1) and the definition \eqref{def Wr}.  

\end{proof}

Next we claim that if $u\in W(B)$ is a (weak) solution of $\LL u=0$ in $B\setminus\Gamma$, we can take test functions in the space $W_0(B\setminus\Gamma)$. That is, 
\begin{equation}\label{eq testfucW0}
    \int_B\A\nabla u\cdot\nabla\vp\, dm=0
   \ \text{ for every } \vp\in W_0(B\setminus\Gamma).
\end{equation}
In fact, 
since $\vp\in W_0(B\setminus\Gamma)$ we can find
a sequence $\set{\vp_k}$ 
in $C_0^\infty(B\setminus\Gamma)$ that converges to $\vp$ in $W(B\setminus\Gamma)$. Then 
\begin{multline*}
    \abs{\int_{B}\A\nabla u\cdot\nabla\vp_k dm - \int_{B}\A\nabla u\cdot\nabla\vp dm }\\
    \le \mu_0\br{\int_B\abs{\nabla u}^2dm}^{1/2}\br{\int_B\abs{\nabla\vp_k-\nabla\vp}^2dm}^{1/2}.
\end{multline*}
The right-hand side is finite and vanishes as $k$ go to infinity. So \eqref{eq testfucW0} follows from taking limits.

Let us also discuss 
the trace on $\bdy(B\setminus\Gamma)=\bdy B\cup (\Gamma\cap B)$. 
Since $W(B)$ is a subset of $W_r(B)$, for $u\in W(B)$, its trace 
 $Tu$  
on $B\cap\Gamma$ can be defined by \eqref{eq trace0} for almost every $x\in B\cap\Gamma$, and $Tu\in L^1_{loc}(B\cap\Gamma,dx)$. Moreover, by slightly modifying the proof of \cite{david2017elliptic}, 
Theorem 3.4, one can show that 
\[
\norm{Tu}_{L^2(B\cap\Gamma,dx)}\lesssim \norm{u}_{L^1(B)}+\norm{\nabla u}_{L^2(B,w)}.
\]
For $u\in W(B)$, we can define its trace on $\bdy B$ by
\[
Tu(X):=\lim_{r\to 0}\fint_{B_r(X)\cap B}u(Y)dY \qquad\text{for }X\in\bdy B,
\]
and one can show
that $\norm{Tu}_{L^2(\bdy B)}\lesssim\norm{u}_{L^1(B)}+\norm{\nabla u}_{L^2(B,w)}$.
Alternatively, since we proved that $W(B)\subset W^{1,2}(B)$, the trace theorem for Sobolev spaces applies.  We remark that in \cite{david2020elliptic}, a trace theorem is developed in a much more general setting and is different from what we have discussed here. But for the purposes of this paper, this simpler approach suffices.

\subsection{Decay estimates for the non-affine part of solutions}\label{subsec Ju0}

We want to show that for a solution of $\LL_0 u=0$ that vanishes on $\Gamma=\Real^d$, its non-affine part $J_u(x,r)$ decreases in $r$. 
In the case when $d=n-1$, this property can be obtained from Moser estimates for solutions on the boundary. We state it in $T_1=T(0,1)$, 
for the constant coefficient operator $L_0$ that was defined in \eqref{def L0},
to simplify the notation.

\begin{lem}[$d=n-1$ case, \cite{DLMgreen}, Lemma 3.4 ]\label{lem u0-lambda t d=n-1}
Let $u\in W^{1,2}(T_1)$ be a solution of $L_0u=0$ in $T_1$ with $u=0$ on $\Delta_1$.
Then there exists some constant $C$ depending only on $d$ and $\mu_0$, such that for $0<r<1/2$,
\begin{equation}\label{eq u0-lambda t}
    \fint_{T_r}\abs{\nabla\br{u(x,t)-\lambda_{r}(u)\, t}}^2dxdt\le C r^2\fint_{T_1}\abs{\nabla (u(x,t)-\lambda_1(u)\, t)}^2dxdt,
\end{equation}
where $\lambda_{r}(u)=\fint_{T_{r}}\partial_s u(y,s)dyds$.\footnote{Note that we are using the same notation $\lambda_r(u)$ to denote different quantities in $d=n-1$ and $d<n-1$.}
\end{lem}

The way we show this decay estimate is by controlling 
the non-affine part of the solution in $T_r$ by the oscillation of the derivative of some solution in $T_r$, which is further controlled by the energy of the solution in $T_1$ multiplied by $r^2$. The bound on the oscillation of the first derivative of solutions is essentially a consequence of estimates for the {\it second} derivatives of solutions. However, when $d<n-1$, we do not have a good estimate for the second derivatives because the coefficients involve $\abs{t}^{-n+d+1}$, which is singular on the boundary. Fortunately, we still have an analogue of Lemma \ref{lem u0-lambda t d=n-1} in the case of $d<n-1$. 
The first step is to show that solutions of $\LL_0u=0$ with a vanishing trace on $\Gamma$ are
roughly  Lipschitz in $t$ near the boundary. 
To be precise, we have the following.

\begin{lem}\label{lem Lip}
Let $B$ be a ball centered on $\Gamma$ and let $u\in W_r(2B)$ be a solution of $\LL_0u=0$ in $2B\setminus\Gamma$, 
 with $Tu=0$  on $\Gamma\cap 2B$. 
 Then there is some constant $C>0$ depending only on $d,n$ and $\mu_0$, such that 
\begin{equation}\label{eq u Lip}
    \abs{u(x,t)}\le C\br{\frac{1}{m(B)}\int_{B}\abs{\nabla u}^2dm}^{1/2}\abs{t}, \qquad\text{for all }\, (x,t)\in B.
\end{equation}
\end{lem}

\begin{proof}
Observe that if $u$ is {\it nonnegative}, then \eqref{eq u Lip} simply follows from the comparison principle and the fact that $\abs{t}$ is a solution of $\LL_0\abs{t}=0$ that vanishes on $\Gamma$. 
In fact, by the comparison principle (Lemma \ref{lem cmpsn}), for $(x,t)\in B\setminus\Gamma$,
\[
\frac{u(x,t)}{\abs{t}}\le 
 C \frac{ u(X_B)}{r(B)}
\le C\br{\frac{1}{m(B)}\int_B\abs{\nabla u}^2dm}^{1/2},
\]
where $X_B$ is a corkscrew point for $B$, 
$r(B)$ denotes its radius, 
and the second inequality is due to Lemma \ref{lem corkscrew}.

If $u$ changes signs in $2B$, we write $u=u_1-u_2$, with 
$u_1=\sup\set{u,0}$ and $u_2=\sup\set{-u,0}$. 
Notice that by Lemma \ref{lem WB} (3), $u\in W(B)$. 
Then by \cite{david2017elliptic},  Lemma~6.1, $u_i\in W(B)$ for $i=1,2$, with 
\[
\int_B\abs{\nabla u_i}^2dm\le\int_B\abs{\nabla u}^2dm, \quad\text{and } Tu_i=0 \text{  on }\Gamma\cap B, \quad i=1,2.
\] 
Moreover, the H\"older continuity of solutions (see \cite{david2017elliptic}, 
Lemma 8.8 and Lemma 8.16) implies that $u\in C(\overline{B})$, and thus $u_i\in C(\overline{B})$ for $i=1,2$. 

We want to look at the solutions $v_i$ to $\LL_0 v_i=0$ in $B\setminus\Gamma$, 
with data $u_i$ on $\bdy(B\setminus\Gamma)$ (and in a suitable weak sense). 
First, the nonhomogeneous problem $\LL_0\wt v_i=-\LL_0 u_i$ in $B\setminus\Gamma$ has a unique solution $\wt v_i\in W_0(B\setminus\Gamma)$ due to the 
Lax-Milgram Theorem. Setting $v_i=\wt v_i+u_i$, one sees that $v_i\in W(B)$ and verifies 
\begin{equation}\label{eq Dir}
    \begin{cases}
    \LL_0v_i=0 \qquad \text{in }B\setminus\Gamma,\\
    v_i-u_i\in W_0(B\setminus\Gamma).
    \end{cases}
\end{equation}
We claim that the $W(B)$ seminorm of $v_i$ is 
controlled by that of $u_i$. To see this, take $v_i-u_i$ as a test function for $\LL_0 v_i=0$, 
which is allowed because $v_i\in W(B)$ and $v_i-u_i\in W_0(B\setminus\Gamma)$ (see the remark around \eqref{eq testfucW0}). Then
\[
\int_B\A_0\nabla v_i\cdot\nabla(v_i-u_i)dm=0.
\]
Therefore, using the ellipticity conditions and the Cauchy-Schwarz inequality, 
\begin{multline*}
     \int_B\abs{\nabla v_{i}}^2dm \le\mu_0\int_B \A_0\nabla v_i\cdot\nabla v_i\, dm = \mu_0\int_B \A_0\nabla v_i\cdot\nabla u_i\, dm\\
    \le\mu_0^2\br{\int_B\abs{\nabla v_i}^2dm}^{1/2}\br{\int_B\abs{\nabla u_i}^2dm}^{1/2},
\end{multline*}
   which implies that
\begin{equation}\label{eq E(v) Lip bd}
   \int_B\abs{\nabla v_i}^2dm\le\mu_0^4\int_B\abs{\nabla u_i}^2dm\le\mu_0^4\int_B\abs{\nabla u}^2dm, \qquad i=1,2. 
\end{equation}

Next, $v_i$ is nonnegative in $B$, for $i=1,2$. To see this, we first show that $v_i$ is continuous in $\overline{B}$.
Since $Tu_i=Tv_i=0$ on $\Gamma\cap B$, the Poincar\'e inequality implies that their weighted $L^2(B,w)$ norm is controlled by their $W(B)$ seminorm. Therefore, both of them belong to the weighted Sobolev space $W^{1,2}(B,w)$. In particular, $u_i\in W^{1,2}(B,w)\cap C(\overline{B})$. Notice that $w(t)$ is an $A_2$ weight with respect to the Lebesgue measure on $\Rn$. That is, there holds
\[
\sup_{B\subset\Rn}\br{\frac{1}{\abs{B}}\int_B\abs{t}^{-n+d+1}dxdt}\br{\frac{1}{\abs{B}}\int_B\abs{t}^{n-d-1}dxdt}<\infty.
\]
So we can apply \cite{heinonen2006nonlinear}, Theorems 6.27 and 6.31, 
to get that for any $X\in\bdy B$, $\lim_{X\rightarrow X_0}v_i(X)=u_i(X_0)$. 
This takes care of continuity on $\bdy B$, so it remains to treat the interior and $\Gamma\cap B$. 
But since $Tv_i=0$ on $\Gamma\cap B$, H\"older estimates for solutions 
(\cite{david2017elliptic} Lemma 8.8 and Lemma 8.16) guarantee that $v_i\in C(B)$. 
So we conclude that $v_i\in C(\overline{B})$, $i=1,2$. Next, we show $v_i\ge 0$ in $B$, using a standard argument.
Set $v^i_\epsilon=\min\set{v_i,-\epsilon}+\epsilon$. Then $v^i_\epsilon\le0$ in $\overline{B}$. Since $v_i\in C(\overline{B})$ is nonnegative on $\bdy(B\setminus\Gamma)$, $v^i_\epsilon$ is compactly supported in $B\setminus\Gamma$. Moreover, 
\begin{equation}\label{eq gradveps}
    \nabla v^i_\epsilon=\begin{cases}
\nabla v_i &\quad v_i<-\epsilon\\
0 &\quad v_i\ge -\epsilon.
\end{cases}
\end{equation}
We take $v^i_\epsilon$ as a test function and get
\[
0=\int_B \A_0\nabla v\cdot\nabla v_\epsilon dm
=\int_B \A_0\nabla v^i_\epsilon\cdot\nabla v^i_\epsilon dm
\ge \mu_0^{-1}\int_{B}\abs{\nabla v^i_\epsilon}^2dm.
\]
This implies that $\nabla v^i_\epsilon=0$ a.e. in $B$, and,
since it is compactly supported in $B \sm \Gamma$, we get that
$v^i_\varepsilon = 0$ a.e. and $v_i \geq -\varepsilon$ in $B$.
Since $\epsilon>0$ is arbitrary, we obtain $v_i\ge 0$ in $B$ for $i=1,2$, as desired.

Now we can apply the result for {\it nonnegative} solutions to $v_i$, and use \eqref{eq E(v) Lip bd} to conclude that
\begin{equation}\label{eq vi}
    v_i(x,t)\le C\br{\frac{1}{m(B)}\int_B\abs{\nabla u}^2dm}^{1/2}\abs{t}, \qquad\text{for }(x,t)\in B, \qquad i=1,2.
\end{equation}
Finally, let $v=v_1-v_2$. Then $v=u$ on $\bdy(B\setminus\Gamma)$ (both in pointwise sense and in $W_0(B\setminus\Gamma)$ sense), and so the uniqueness of the solution implies $u=v$ in $B$. The 
desired estimate for $\abs{u(x,t)}$ follows from \eqref{eq vi} and the fact that $\abs{u}\le v_1+v_2$ in $B$.
\end{proof}

Now we derive the decay estimate for solutions of $\LL_0 u=0$, which is an analogue of Lemma \ref{lem u0-lambda t d=n-1} for $d<n-1$. By the translation and dilation invariance of the problem, we only need to consider the problem on the unit ball. We shall use $J_u(r)$ to denote $J_u(0,r)$. Similarly, $E_u(r)$ and $\beta_u(r)$ are shorthand for $E_u(0,r)$ and $\beta_u(0,r)$, respectively.

\begin{lem}[Key lemma]\label{lem Ju0 decay} For any $\theta_0\in (0,1)$, there exists $r_0=r_0(n,d,\mu_0,\theta_0)\in(0,1)$ such that for any solution $u\in W(B_1)$ of $\LL_0u=0$ in $B_1\setminus\Gamma$, with $Tu=0$  on $\Gamma\cap B_1$, there holds
\begin{equation}\label{eq Ju0 decay}
    J_u(r)\le\theta_0J_u(1), \qquad\text{for}\quad 0<r\le r_0.
\end{equation}
\end{lem}

\begin{proof}
We first show that for any $r_1\in (0,1)$ and any $\theta_1\in (0,1)$, there exists $r_0=r_0(\theta_1,r_1,n,d)<r_1$, such that for any solution $u\in W(B_1)$ of $\LL_0u=0$ in $B_1\setminus\Gamma$, with $Tu=0$  on $\Gamma\cap B_1$, there holds
\begin{equation}\label{eq step1}
    \frac{1}{m(B_r)}\int_{B_r}\abs{\nabla\br{u-\wt{u}_{\theta}}}^2dm \le \frac{\theta_1}{m(B_{r_1})}\int_{B_{r_1}}\abs{\nabla\br{u-\wt{u}_{\theta}}}^2dm\qquad\text{for }r\le r_0,
\end{equation}
where $\wt u_\theta$ is defined in \eqref{eqdef wtutheta}.
We prove \eqref{eq step1} by contradiction. If the statement is not true, then there is a $\theta_1\in(0,1)$, 
a sequence of operators $\LL_0^{(k)} \in \cA_0(\mu_0)$,
a sequence $\set{r_k}_{k=1}^\infty$ decreasing to $0$, and a sequence of solutions 
$\set{u^{(k)}}_{k=1}^\infty\subset W(B_1)$ verifying 
$\LL_0^{(k)}u^{(k)}=0$ 
in $B_1\setminus\Gamma$ and $Tu^{(k)}=0$  on $B_1\cap\Gamma$, such that
\begin{equation}\label{eq u-utheta ifnot}
    \frac{1}{m(B_{r_k})}\int_{B_{r_k}}\abs{\nabla(u^{(k)}-\wt{u}^{(k)}_{\theta})}^2dm>\frac{\theta_1}{m(B_{r_1})}\int_{B_{r_1}}\abs{\nabla(u^{(k)}-\wt{u}^{(k)}_{\theta})}^2dm,
\end{equation} 
for $k=1,2,\dots$.
Define 
\[
v_k=\frac{u^{(k)}-\wt{u}^{(k)}_{\theta}}{\br{m(B_{r_1})^{-1}\int_{B_{r_1}}\abs{\nabla \br{u^{(k)}-\wt{u}^{(k)}_{\theta}}}^2dm}^{1/2}}.
\]
Notice that we do not need to worry about the denominator being equal to $0$ because in that case, 
both sides of \eqref{eq u-utheta ifnot} are $0$, making the inequality false. 
By Lemma~\ref{lem utheta sol of L0}, $\wt{u}^{(k)}_{\theta}$ 
verifies $\LL_0^{(k)}\wt{u}^{(k)}_{\theta}=0$,  
and thus $v_k$ verifies $\LL_0^{(k)} v_k=0$ 
in $B_1\setminus\Gamma$, with $Tv_k=0$  on $B_1\cap\Gamma$. 
Moreover, $v_k$ is constructed in a way that guarantees the following properties: 
\begin{align}
   \fint_{\d B(0,r)}v_k\,d\omega=0 \ \text{ for } 0 < r \leq 1,  
    \nonumber\\
    \frac{1}{m(B_{r_1})}\int_{B_{r_1}}\abs{\nabla v_k}^2dm=1,\label{eq vknrml}\\
    \frac{r_k^{-2}}{m(B_{2r_k})}\int_{B_{2r_k}}\abs{v_k}^2dm\ge\theta_1/C,\nonumber
\end{align}
where the last inequality follows from \eqref{eq u-utheta ifnot} and the Caccioppoli inequality on the boundary.

Set $V_k(X):=\frac{1}{r_k}v_k(r_kX)$. Then $\LL_0^{(k)}V_k=0$ 
in $B_{1/r_k}\setminus\Gamma$, with $TV_k=0$ on $B_{1/r_k}\cap\Gamma$. Moreover,
\begin{align}
   \fint_{\d B(0,r)}V_k\,d\omega=0 \ \text{ for } 0 < r \leq 1/r_k,  \label{eq SVk=0}\\
      m(B_2)^{-1}\int_{B_{2}}\abs{V_k}^2dm\ge\theta_1/C.\label{eq VkL2 lwbd}
\end{align}
Notice that \eqref{eq VkL2 lwbd} implies that there exists $(x_k,t_k)\in B_2$ such that  
\begin{equation}\label{eq Xk lwbd}
    \abs{V_k(x_k,t_k)}\ge\sqrt{\theta_1/C}.
\end{equation}
Observe that by \eqref{eq vknrml},
\begin{equation}\label{eq E(wt vk)=1}
    \frac{1}{m(B_{\frac{r_1}{r_k}})}\int_{B_{\frac{r_1}{r_k}}}\abs{\nabla V_k}^2dm=1.
\end{equation}
By \eqref{eq E(wt vk)=1} and Lemma \ref{lem Lip}, there is some constant $c>0$ depending only on $d,n$ and $\mu_0$, such that
\begin{equation}\label{eq wtvk Lip}
    \abs{V_k(x,t)}\le c\abs{t} \quad\text{for all}\quad (x,t)\in B_{\frac{r_1}{2r_k}}.
\end{equation}
Now \eqref{eq Xk lwbd} and \eqref{eq wtvk Lip} imply that the $t_k$ in \eqref{eq Xk lwbd} has to satisfy 
\begin{equation}\label{eq tk}
    2\ge\abs{t_k}\ge C'\theta_1^{1/2}.
\end{equation} 
Moreover, on any compact set in $\Rn$, \eqref{eq wtvk Lip} implies that the sequence 
$\{V_k\}_{k=1}^\infty$ is uniformly bounded, and the regularity of solutions implies that $\{V_k\}_{k=1}^\infty$ is equicontinuous. Therefore, there is a subsequence of $\{V_k\}$, still denoted by $\{V_k\}$, converges pointwise to a $V_{\infty}$. 
We can also find a limit $\LL_0 \in \cA_0(\mu_0)$ of the $\LL_0^{(k)}$, and it is 
easy to verify that $V_{\infty}\in W_r(\Rn)$ is a solution of $\LL_0V_\infty=0$ in $\Rn\setminus \Gamma$, with $V_{\infty}(x,0)=0$ on $\Gamma$. 
For sure there is a convergent subsequence of $\set{(x_k,t_k)}$ in $B_2$; let us denote the limit point by 
$(x_{\infty},t_{\infty})\in B_2$. Then by \eqref{eq Xk lwbd} and \eqref{eq tk},
\begin{equation}\label{eq vinty positive}
    \abs{t_{\infty}}>C'\theta_1^{1/2}, \quad\quad   \abs{V_{\infty}(x_{\infty},t_{\infty})}\ge\sqrt{\theta_1/C}.
\end{equation} 
By \eqref{eq wtvk Lip} (and the fact that $r_k$ tends to $0$), 
$2c\abs{t}-V_{\infty}(x,t)>0$ everywhere. 
So $2c\abs{t}-V_{\infty}(x,t)\in W_r(\Rn)$ is a positive solution in $\Rn\setminus\Gamma$ that vanishes on $\Gamma$. On the other hand, $\abs{t}\in W_r(\Rn)$ is also a positive solution in $\Rn\setminus\Gamma$ that vanishes on $\Gamma$. Therefore, we can apply the Corollary \ref{cor cmpsn} to $2c\abs{t}-V_{\infty}(x,t)$ and $\abs{t}$, and obtain
\[
\abs{\frac{2c\abs{t}-V_{\infty}(x,t)}{\alpha\abs{t}}-1}\le C\br{\frac{\abs{(x,t)-(0,1)}}{R}}^\gamma \qquad\text{for all }\, R\ge2,
\]
where $\alpha=2c-V_{\infty}(0,1)>0$. Letting $R\to\infty$ one sees that
\(
2c\abs{t}-V_{\infty}(x,t)=\alpha\abs{t}\), and thus \(V_{\infty}(x,t)=\alpha'\abs{t}\) for $(x,t)\in\Rn$.
Thanks to \eqref{eq vinty positive}, $\alpha'\neq0$.
Therefore, 
\(\fint_{S^{n-d-1}}V_{\infty}d\omega\neq0\), which is impossible since \eqref{eq SVk=0} holds for all $k$. This proves \eqref{eq step1}.

Now we show \eqref{eq Ju0 decay}.  Fix $r_1\in (0,1/2)$ and $\theta_1\in(0,1)$ to be determined later, 
and let $r_0= r_0(\theta_1,r_1,n,d)<r_1$ 
be as in \eqref{eq step1}. Then for any $0<r\le r_0$, we write
\begin{multline*}
    J_u(r)=\frac1{m(B_r)}\int_{B_r}\abs{\nabla(u(x,t)-\lambda_r(u)\abs{t})}^2w(t)\,dxdt\\
    \le \frac2{m(B_r)}\int_{B_r}\abs{\nabla(u-\wt{u}_{\theta})}^2dm+\frac2{m(B_r)}\int_{B_r}\abs{\nabla (\wt{u}_{\theta}-\lambda_r(u)\abs{t})}^2w(t)\,dxdt,
    \end{multline*}
where we recall from \eqref{def lambda} that
that \[\lambda_r(u)=\lambda_{0,r}(u)=\frac{1}{m(B_r)}\int_{B_r}\frac{\nabla_t u(x,t)\cdot t}{\abs{t}}w(t)\,dxdt.\] 
Apply \eqref{eq step1} to get
    \begin{multline*}
    J_u(r)\le \frac{2\theta_1}{m(B_{r_1})}\int_{B_{r_1}}\abs{\nabla(u-\wt{u}_{\theta})}^2dm+\frac2{m(B_r)}\int_{B_r}\abs{\nabla (\wt{u}_{\theta}-\lambda_r(u)\abs{t})}^2w(t)\,dxdt.
    \end{multline*}
    Inserting $\lambda_{r_1}(u)\abs{t}$ in the first integral on the right-hand side, 
    \begin{multline}
        J_u(r)\le \frac{4\theta_1}{m(B_{r_1})}\int_{B_{r_1}}\abs{\nabla (u-\lambda_{r_1}(u)\abs{t})}^2w(t)\,dxdt\\
    +\frac{4\theta_1}{m(B_{r_1})}\int_{B_{r_1}}\abs{\nabla (\wt{u}_{\theta}-\lambda_{r_1}(u)\abs{t})}^2w(t)\,dxdt\\
      +\frac2{m(B_r)}\int_{B_r}\abs{\nabla (\wt{u}_{\theta}-\lambda_r(u)\abs{t})}^2w(t)\,dxdt. \label{eq Ju0r split}
\end{multline}
We estimate the the last two terms in \eqref{eq Ju0r split} using decay estimates for the case $d=n-1$. First, changing to polar coordinates as in \eqref{eq mBandT}, one sees that
\begin{multline}\label{eq utheta-lambda 1}
     \frac1{m(B_r)}\int_{B_r}\abs{\nabla (\wt{u}_{\theta}(x,t)-\lambda_r(u)\abs{t})}^2w(t)\, dxdt\\
        =\frac{1}{m(B_r)}\int_{\abs{x}\le r}\int_0^{\sqrt{r^2-\abs{x}^2}}\abs{\nabla_{x,\rho}\br{u_{\theta}(x,\rho)-\lambda_r(u)\rho}}^2\br{\int_{S^{n-d-1}}d\omega}d\rho dx\\
    =\fint_{T_r}\abs{\nabla_{x,\rho}(u_{\theta}(x,\rho)-\lambda_r(u)\rho)}^2d\rho dx.
\end{multline}
Recall from Lemma \ref{lem lambda(u) in polar} that $\lambda_r(u)=\fint_{T_r}\partial_\rho u_\theta(y,\rho)dyd\rho$. Since $u_\theta$ verifies $L_0u_\theta=0$ (see Lemma \ref{lem utheta sol of L0}), we can apply Lemma \ref{lem u0-lambda t d=n-1} to $u_\theta$ and get 
\begin{align}\label{eq TrT1}
   \fint_{T_r}\abs{\nabla_{x,\rho}(u_{\theta}(x,\rho)-\lambda_{r}(u)\rho)}^2d\rho dx
   \le Cr^2\fint_{T_1}\abs{\nabla_{x,\rho}(u_{\theta}(x,\rho)-\lambda_1(u)\rho)}^2d\rho dx.
    \end{align}
Notice that 
\begin{multline*}
     \abs{\nabla_{x,\rho}\br{u_\theta(x,\rho)-\lambda_1(u)\rho}}^2=\abs{\nabla_{x,t}\br{\wt{u}_\theta(x,t)-\lambda_1(u)\abs{t}}}^2\\
    \le\fint_{S^{n-d-1}}\abs{\nabla_{x,t}\br{u(x,\abs{t}\omega)-\lambda_1(u)\abs{t}}}^2d\omega.
\end{multline*}
By a computation similar to that in the proof of Lemma \ref{lem spcs utheta}, this yields 
\begin{align*}
 \fint_{T_1}\abs{\nabla_{x,\rho}(u_{\theta}(x,\rho)-\lambda_{1}(u)\rho)}^2d\rho dx\le \frac1{m(B_1)}\int_{B_1}\abs{\nabla(u(x,t)-\lambda_1(u)\abs{t})}^2w(t)\,dxdt.
\end{align*}
Combining this with \eqref{eq TrT1} and \eqref{eq utheta-lambda 1}, we obtain
\begin{multline}\label{eq utheta-r final}
     \frac1{m(B_r)}\int_{B_r}\abs{\nabla (\wt{u}_{\theta}(x,t)-\lambda_{r}(u)\abs{t})}^2dw(t)\, dxdt\\
    \le\frac{Cr^2}{m(B_1)}\int_{B_1}\abs{\nabla(u(x,t)-\lambda_1(u)\abs{t})}^2w(t)\, dxdt.
\end{multline}
Now we return to the first term in the right-hand side of \eqref{eq Ju0r split}. Since $\lambda_r$ is a minimizer (see \eqref{eq min lambda}),
\begin{align*}
   \frac{4\theta_1}{m(B_{r_1})}\int_{B_{r_1}}\abs{\nabla (u-\lambda_{r_1}(u)\abs{t})}^2w(t)
   \le \frac{4\theta_1}{m(B_{r_1})}\int_{B_{r_1}}\abs{\nabla (u-\lambda_1(u)\abs{t})}^2w(t).
   \end{align*}
   Enlarging the ball, the right-hand side is bounded by \[
    \frac{4\theta_1}{r_1^{d+1}}\frac{1}{m(B_1)}\int_{B_1}\abs{\nabla(u(x,t)-\lambda_1\abs{t})}^2w(t)\, dxdt.
\]
This estimate, together with \eqref{eq utheta-r final} and \eqref{eq Ju0r split}, gives 
\begin{multline*}
    \frac1{m(B_r)}\int_{B_r}\abs{\nabla(u(x,t)-\lambda_r(u)\abs{t})}^2w(t)\, dxdt\\
    \le\br{\frac{4\theta_1}{r_1^{d+1}}+C\theta_1r_1^2+Cr_1^2}\frac{1}{m(B_1)}
    \int_{B_1}\abs{\nabla(u-
   \lambda_1(u) \abs{t})}^2w(t)\, dxdt
\end{multline*}
for $0<r\le r_0$.
Now we only need to choose $\theta_1$ and $r_1$ properly. Let for instance,  $\theta_1=r_1^{d+2}$, and then choose $r_1=r_1(\theta_0,n,d,\mu_0)\in (0,1)$ sufficiently small so that $4r_1+Cr_1^{d+4}+Cr_1^2\le\theta_0$. Recall that $r_0$ is determined by $\theta_1$ and $r_1$, and thus depends only on $\theta_0,n,d$, and $\mu_0$. This completes the proof of the key lemma. 
\end{proof}

Ultimately, we want to derive a decay estimate for the {\it normalized} non-affine part of the local energy of $u$, i.e. $\beta_r(u)$. So we need to compare the local energy of positive solutions of $\LL_0 u=0$ for different scales.

\begin{lem}\label{lem Eu0 lwbd}
Let $u\in W(B_1)$ be a positive solution of $\LL_0u=0$ in $B_1\setminus\Gamma$ with $Tu=0$  on $\Gamma\cap B_1$. Then
\[
E_u(r)\ge  C(1-C' r^2) E_u(1) \qquad\text{for } 0<r<\frac{1}{2},
\]
where $C$ and $C'$ are positive constants depending only on $d,n$ and $\mu_0$.
\end{lem}

\begin{proof}
Recall that by Lemma \ref{lem utheta sol of L0}, $u_\theta$ is a solution of the $(d+1)$ dimensional operator $L_0$, and that by Lemma \ref{lem lambda(u) in polar}, $\lambda_r(u)=\fint_{T_r}\partial_\rho u_\theta(x,\rho)dxd\rho$. So by the boundary regularity of the 
solutions of constant-coefficient operator $L_0$ in $\R_+^{d+1}$ (see \cite{DLMgreen} Lemma 2.10), 
\begin{align}\label{eq lambdar-lambda0}
    \abs{\lambda_r(u)-\lambda_s(u)}
    &\le\underset{T_r}{\osc}\partial_\rho u_\theta
    \le C r \br{\fint_{T_1}\abs{\nabla_{x,\rho}u_\theta(x,\rho)}^2dxd\rho}^{1/2}\nonumber\\
    &\le C r\br{\frac{1}{m(B_1)}\int_{B_1}\abs{\nabla u(x,t)}^2w(t)\, dxdt}^{1/2}
\end{align}
for 
$0 < s < r  <1/2$.  Hence $\lambda_0(u) = \lim_{s \to 0} \lambda_s(u)$ exists, and 
since we even have a bound on $\underset{T_r}{\osc}\partial_\rho u_\theta$, 
we see that $\lambda_0(u)=\partial_\rho u_\theta(0,0)$.
Since $\LL_0\abs{t}=0$, we can apply the comparison principle (Lemma \ref{lem cmpsn}) to get that \[
\frac{u(x,t)}{\abs{t}}\approx 
u(x,t_0) \qquad\text{for all }(x,t)\in B_{1/2} \quad\text{ and any  }t_0 \text{ such that }\abs{t_0}=\frac{1}{2}.\]
So by Lemma \ref{lem corkscrew}, 
\[\frac{u(x,t)}{\abs{t}}\approx\br{\frac{1}{m(B_1)}\int_{B_1}\abs{\nabla u}^2dm}^{1/2}\qquad\text{for all }\, (x,t)\in B_{1/2},
\]
which implies that 
\begin{equation}\label{eq utheta/rho}
    \frac{u_\theta(x,\rho)}{\rho}=\frac{\fint_{S^{n-d-1}}u(x,\rho\omega)d\omega}{\rho}\approx\br{\frac{1}{m(B_1)}\int_{B_1}\abs{\nabla u}^2dm}^{1/2}
\end{equation}
for any $(x,\rho)\in T_{1/2}$. 
Letting $\rho\to0$, this yields a bound
\[
\lambda_0(u)=\partial_\rho u_\theta(0,0)\gtrsim\br{\frac{1}{m(B_1)}\int_{B_1}\abs{\nabla u}^2dm}^{1/2}.
\]
Combining it with \eqref{eq lambdar-lambda0}, we get
\begin{align*}
    \frac{1}{m(B_r)}\int_{B_r}\abs{\nabla u}^2dm&\ge\lambda_r^2(u)\ge\frac{\lambda_0^2(u)}{2}-\br{\lambda_r(u)-\lambda_0(u)}^2\\
    &\ge \br{C-C'r^2}\frac{1}{m(B_1)}\int_{B_1}\abs{\nabla u}^2dm,
\end{align*}
as desired.
\end{proof}

\section{Extension to a general operator $\LL$}\label{sec L}
\subsection{Decay estimates}\label{subsec L}
In this subsection, we shall follow the approach that is used in \cite{DLMgreen} to obtain a decay estimate for the normalized non-affine part of the energy of solutions of $\LL u=0$. Namely, we shall approximate $\beta_u(r)$ by $\beta_{u_0}(r)$, with $u_0$ verifying $\LL_0u_0=0$, and show that the error is a Carleson measure. Since the strategy is the same as in the $d=n-1$ setting, we shall focus less on motivation but more on technical details that are different from the co-dimension 1 case. For the same reason, many proofs will be omitted if they can be borrowed from \cite{DLMgreen} without substantial changes. 

We start with comparing solutions of $\LL u=0$ and solutions of $\LL_0 u_0=0$ with the same boundary data. The following two lemmas hold for any matrix $\A_0$ satisfying the ellipticity conditions \eqref{cond ellp}. Ultimately, we will apply them to $\A_0\in\cA_0(\mu_0)$. 

\begin{lem}\label{lem comp u u0}
Let $u\in W(B_1)$ be a solution to $\LL u=0$ in $B_1\setminus\Gamma$ with $Tu=0$ on $\Gamma\cap B_1$. Let $u^0\in W(B_1)$ be a solution to $\LL_0u^0=0$ in $B_1\setminus\Gamma$ with $u^0-u\in W_0(B_1\setminus\Gamma)$. Then there is a constant $C>0$ depending only on the ellipticity constant $\mu_0$, $d$ and $n$, such that 
\begin{multline}\label{eq min}
    \int_{B_1}\abs{\nabla (u-u^0)}^2dm\\
    \le \mu_0^2\min\set{\int_{B_1}\abs{\A-\A_0}^2\abs{\nabla u}^2dm,\int_{B_1}\abs{\A-\A_0}^2\abs{\nabla u^0}^2dm}.
\end{multline}
\end{lem}

\begin{proof}
 First of all, the existence of $u^0$ is guaranteed by the 
 Lax-Milgram Theorem. Taking $u-u^0\in W_0(B_1\setminus\Gamma)$ as a test function in the equation $\LL u=0$, using ellipticity conditions and Young's inequality, we can get 
 \begin{multline*}
      \mu_0^{-1}\int_{B_1}\abs{\nabla(u-u^0)}^2dm\le\int_{B_1} A\nabla(u-u^0)\cdot\nabla(u-u^0)dm\\
    = - \int_{B_1} A\nabla u^0\cdot\nabla(u-u^0)dm
    =\int_{B_1}(A_0-A)\nabla u^0\cdot\nabla(u-u^0)dm\\
    \le \frac{\mu_0}{2}\int_{B_1}\abs{A-A_0}^2\abs{\nabla u^0}^2dm +\frac{1}{2\mu_0}\int_{B_1}\abs{\nabla(u-u^0)}^2dm.
 \end{multline*}
 This yields 
\[
\int_{B_1}\abs{\nabla(u-u^0)}^2dm\le \mu_0^2\int_{B_1}\abs{\A-\A_0}^2\abs{\nabla u^0}^2dm.
\]
Interchanging the roles of $u$ and $u^0$, and $\A$ and $\A_0$, we also obtain the other bound.
\end{proof}

\begin{lem}\label{lem u=u^0 d<n-1}
Let $u$ and $u^0$ be as in Lemma \ref{lem comp u u0}. Then 
\begin{equation*}
    C^{-1}\int_{B_1}\abs{\nabla u^0}^2dm\le\int_{B_1}\abs{\nabla u}^2dm\le C\int_{B_1}\abs{\nabla u^0}^2dm,
 \end{equation*}
 where $C=\mu_0^4$.
\end{lem}

The triangle inequality would almost give this directly; the proof (with $C=\mu_0^4$)
is the same as when $d=n-1$ and is thus omitted; see \cite{DLMgreen}, Lemma 3.13.

Define
\begin{equation} \label{def gamma}
\gamma(x,r) = \inf_{\A_0 \in \cA_0(\mu_0)}
\bigg\{m(B(x,r))^{-1}\int_{(y,t) \in B(x,r)} |\A(y,t)-\A_0|^2w(t)\,dydt \bigg\}^{1/2}.
\end{equation}
Notice that the domain of integration is larger than what we have in \eqref{def alpha}.

\begin{lem}\label{lem gamma}
If the matrix-valued function $\A$ satisfies the weak DKP condition
of Definition \ref{def wDKP}, with constant $\varepsilon > 0$,
then $\gamma(x,r)^2 \frac{dxdr}{r}$ is a Carleson measure on $\R^{d+1}_+$, with the norm
\begin{equation}\label{3a17}
\norm{\gamma(x,r)^2 \frac{dxdr}{r}}_{\mathcal{C}} 
\leq C \cN(\A) \le C \varepsilon,
\end{equation}
where $\cN(\A) = \norm{\alpha(x,r)^2 \frac{dxdr}{r}}_{\mathcal{C}}$
is as in \eqref{def alpha} - \eqref{eq wDKP}, and 
\begin{equation}\label{3a18}
\gamma(x,r)^2 \leq C \cN(\A) \le C \varepsilon 
\quad \text{ for } (x,r) \in \R^{d+1}_+.
\end{equation}
Here, $C$ depends only on $d$, $n$, and $\mu_0$.
\end{lem}

\begin{proof}
This lemma can be proved quite similarly as the $d=n-1$ case. Here, we only mention some modifications and refer the readers to \cite{DLMgreen}, Section 4.1, for details. 

We want to show $\gamma(x,r)^2\frac{dxdr}{r}$ is a Carleson measure on $\Real^{d+1}_+$. Let $\Delta_0=\Delta(x_0,r_0)$ be given. We claim that we can control $\gamma(x_0,r_0)$ in terms of $\alpha$ as in the case of $d=n-1$. That is, we want to show that
\begin{equation}\label{eq gammax0r0}
    \gamma(x_0,r_0)^2 \leq C \alpha_2(x_0,r_0)^2 
+ C \sum_{m\ge 0} \sigma^{\frac{m}{2}} \fint_{\Delta'_0} \alpha_2(y,\sigma^m r_0)^2 dy,
\end{equation}
where $\sigma=\frac{4}{5}$, and $\Delta_0'=\Delta(x_0,3r_0/2)$.
To this end, for each pair $(x,r)$, choose a $\A_{x,r}\in \cA_0(\mu_0)$ such that
\[
m(W(x,r))^{-1}\int_{W(x,r)}\abs{\A(y,t)-\A_{x,r}}^2w(t)\,dydt=\alpha(x,r)^2.
\]
Let $\A_0=\A_{x_0,r_0}$. Then 
\begin{multline*}
    \gamma(x_0,r_0)^2\le m(B(x_0,r_0))^{-1}\int_{(y,t) \in B(x_0,r_0)} |\A(y,t)-\A_0|^2w(t)dydt\\
    \le 
    \frac{1}{m(B(x_0,r_0))} 
    \int_{y\in\Delta_0}\int_{\abs{t}\le r_0}|\A(y,t)-\A_0|^2w(t)dtdy.
\end{multline*}
Let $Q_0=\set{(x,t):x\in\Delta_0, \abs{t}\le r_0}$. As in the case of $d=n-1$, 
we cut $Q_0$ into horizontal slices $H_m$ associated to 
radii $r_m=\sigma^m r_0$, $m\ge 0$. The only difference is that now these slices are annular regions. That is, $H_m=\set{(x,t):x\in\Delta_0,\, r_{m+1}<\abs{t}\le r_{m}}$. Once we have set this up, \eqref{eq gammax0r0} can be obtained by showing that
\[
\int_{H_m}\abs{\A(y,t)-\A_0}^2w(t)dydt\le C r_m  \alpha_2(x_0,r_0)^2  |\Delta_0|
+ C r_m \int_{\Delta'_0} \Big\{\sum_{j=0}^m \alpha_2(y,r_j) \Big\}^2 dy
\]
as for $d=n-1$. 

Now \eqref{3a17} and \eqref{3a18} can be obtained verbatim from the proof in the $d=n-1$ case, 
since we have \eqref{eq gammax0r0} and both $\gamma$ and $\alpha$ are functions on $\Real^{d+1}_+$.
\end{proof}

The following estimate on $\nabla u$ can be proved similarly to in the case $d=n-1$. One only needs to replace Carleson balls in $\Real^{d+1}_+$ with balls centered on $\Gamma$ in $\Rn$. One needs to use the reverse H\"older estimate Lemma \ref{lem RH}, which gives an exponent greater than 2 that depends only on $d$, $n$ and $\mu_0$. We refer readers to \cite{DLMgreen}, Lemma~3.19, for details of the proof.

\begin{lem}\label{l319}
Let $u\in W_r(B_5)$ be a positive solution to $\LL u=0$ in $B_5\setminus\Gamma$, such that 
$Tu=0$  on $\Gamma\cap B_5$. Choose a matrix $\A_0 \in \cA_0(\mu_0)$
that attains the infimum in the definition \eqref{def gamma} for $\gamma(0,1)$,
and let $u^0$ be the solution from Lemma \ref{lem comp u u0} (with this choice of $\A_0$). Then 
for any $\delta>0$,
\begin{equation} \label{3a20}
\int_{B_1} \abs{\nabla u- \nabla u^0}^2 dm
\leq \br{\delta+C_\delta \gamma(0,1)^2} E_u(1),
\end{equation}
where $C_\delta$ depends on $d$, $n$, $\mu_0$, and $\delta$. 
\end{lem}

We can now derive the decay estimates for 
the non-affine part of solutions $u$. The following is an analogue of Lemma \ref{lem Ju0 decay}, and should be compared to Lemma 3.24 in \cite{DLMgreen} for the case $d=n-1$.

\begin{lem}\label{lem Jur}
Let $u\in W(B_1)$ be a solution to $\LL u=0$ in $B_1\setminus\Gamma$ with $Tu=0$  on $\Gamma\cap B_1$. Then there exist constants $p=p(d,n,\mu_0)\in(2,\infty)$, $C=C(d,n,\mu_0)\in(0,\infty)$ such that for any $\theta_0\in(0,1)$, there exists $r_0=r_0(\theta_0,d,n,\mu_0)\in(0,1/4)$, such that
\begin{equation} \label{Jur}
J_u(r) \leq C \br{\theta_0+K^{\frac{2-p}{2}}r^{-d-1}}J_u(1) 
+ \frac{C_K}{r^{d+1}}\gamma(0,1)^2E_u(1) 
\end{equation}
for any $0<r\le r_0$, and any $K>0$. Here, $C_K$ depends on $K$, and $d,n,\mu_0$.
\end{lem}

\begin{proof}
In what follows, we shall follow rather closely the proof of the $d=n-1$ case,
and refer to \cite{DLMgreen} for an occasional missing detail. 
We  shall choose a $u_0$ verifying $\LL_0u_0=0$, use the decay estimates for $J_{u_0}(r)$ to get a decay estimate for $J_u(r)$ with an error \eqref{eq Jur step1}. Then using some reverse H\"older estimates,
we shall
control the error by terms on the right-hand side of \eqref{eq Ju0 decay}.  

We write $u$ as an affine part plus its complement on $B_1$, i.e.
\[
u(x,t)=v(x,t)+\lambda_1(u)\abs{t}.
\]
Notice that $E_v(1)=J_u(1)$ by 
 the definitions near \eqref{def lambda}, and in addition
\begin{equation}\label{eq lambdaE}
    \lambda_1(u)^2\le \frac{1}{m(B_1)}\int_{B_1}\abs{\nabla_t u}^2w(t)\, dxdt \le E_u(1) 
\end{equation}

Choose a matrix $\A_0$
in the compact set $\cA_0(\mu_0)$, that attains the infimum in the definition \eqref{def gamma} of $\gamma(0,1)$, and let 
$\LL_0=-\divg\br{\A_0 w(t)\nabla}$ as usual.

Now consider the $\LL_0$-harmonic extension $u_0$ of the restriction of $u$ to 
$\bdy (B_{1/2}\setminus\Gamma)$, that is, the unique solution $u_0\in W(B_{1/2})$ to $\LL_0 u_0=0$ in $B_{1/2}\setminus\Gamma$, with $u_0-u\in W_0(B_{1/2}\setminus\Gamma)$. Write 
\begin{equation}\label{ext u0}
 u_0(x,t) = v_0(x,t) + \lambda_1(u) \abs{t}.   
\end{equation}
Since $\LL_0\abs{t}=0$, $v_0\in W(B_{1/2})$ verifies 
\begin{equation}
    \LL_0 v_0=0\quad \text{in  } B_{1/2}\setminus\Gamma \ \text{ and } 
    v_0-v\in W_0(B_{1/2}\setminus\Gamma).
\end{equation}
In particular, $Tv_0=Tv=0$ on $B_{1/2}\cap\Gamma$. 

We claim that for any $\theta_0\in(0,1)$, there exists $r_0\in(0,1/4)$ depending on $\theta_0$, $d$, $n$ and $\mu_0$, and a constant $C$ depending only on $d,n,\mu_0$, such that 
\begin{equation} \label{eq Jur step1}
J_u(r) \leq C\theta_0 J_u(1)
+\frac{C(\theta_0+r^{-d-1})}{m(B_{1/2})}\int_{B_{1/2}}\abs{\A-\A_0}^2\abs{\nabla u_0}^2dm,
\end{equation}
for any $0<r\le r_0$.

To see this, we use the inequality $(a+b+c)^2 \leq 3(a^2 + b^2 + c^2)$ to write
\begin{multline}\label{eq Jur3tm}
 J_u(r)
 \le \frac3{m(B_r)}\int_{B_r}\abs{\nabla(u_0 
    -\lambda_r(u_0)\, t)}^2     \,dm
    + \frac3{m(B_r)}\int_{B_r}\abs{\nabla(u-u_0)}^2dm\\
    +\frac3{m(B_r)}\int_{B_r}\abs{\nabla(\lambda_r(u_0)\abs{t}-\lambda_r(u)\abs{t})}^2 \,dm.
\end{multline}
The last integral can be controlled by the second integral on the right-hand side of 
\eqref{eq Jur3tm}, as follows : 
\begin{multline}\label{eq Jur0.3}
    \frac{1}{m(B_r)}\int_{B_r}\abs{\nabla(\lambda_r(u_0)\abs{t}-\lambda_r(u)\abs{t})}^2 \,dm
    =(\lambda_r(u_0)-\lambda_r(u))^2\\
    =\br{\frac{1}{m(B_r)}\int_{B_r}\frac{\nabla_t(u-u_0)\cdot t}{\abs{t}} dm 
    }^2
    \le\frac{1}{m(B_r)}\int_{B_r}\abs{\nabla(u-u_0)}^2dm.
\end{multline}
 For the second integral on the right-hand side of \eqref{eq Jur3tm}, we enlarge $B_r$ and apply Lemma \ref{lem comp u u0} to get 
\begin{align}
   \frac{1}{m(B_r)}\int_{B_r}\abs{\nabla(u-u_0)}^2dm&\le\frac{r^{-(d+1)}}{m(B_{1/2})}\int_{B_{1/2}}\abs{\nabla(u-u_0)}^2dm\nonumber\\
   &\le \frac{Cr^{-(d+1)}}{m(B_{1/2})}\int_{B_{1/2}}\abs{\A-\A_0}^2\abs{\nabla u_0}^2dm.\label{eq Jur0.2}
\end{align}

Finally, by Lemma \ref{lem Ju0 decay}, for any fixed $\theta_0\in(0,1)$,  there is some $r_0=r_0(\theta_0,n,d,\mu_0)\in(0,1/4)$ such that the first integral in \eqref{eq Jur3tm} is bounded by $\frac{\theta_0}{3}J_{u_0}(1/2)$. On the other hand, the same sort of computation as 
above gives
\begin{multline*}
   J_{u_0}(1/2)
    \le 3J_{u}(1/2)+\frac{3}{m(B_{1/2})}\int_{B_{1/2}}\abs{\nabla(u-u_0)}^2dm
    +3(\lambda_{1/2}(u)-\lambda_{1/2}(u_0))^2\\
    \le 3J_{u}(1/2)+\frac{C}{m(B_{1/2})}\int_{B_{1/2}}\abs{\A-\A_0}^2\abs{\nabla u_0}^2dm.
\end{multline*}
Combining this with \eqref{eq Jur3tm}, \eqref{eq Jur0.3} and \eqref{eq Jur0.2}, we obtain 
\[
J_u(r) \leq \theta_0 J_u(1/2)
+\frac{C(\theta_0+r^{-d-1})}{m(B_{1/2})}\int_{B_{1/2}}\abs{\A-\A_0}^2\abs{\nabla u_0}^2dm,
\]
which is almost \eqref{eq Jur step1}. To show \eqref{eq Jur step1}, we only need to observe that by the minimizing property of $\lambda_{1/2}(u)$ (see \eqref{eq min lambda}),  
\[
J_u(1/2)\le m(B_{1/2})^{-1}\int_{B_{1/2}}\abs{\nabla(u(x,t)-\lambda_1(u)\abs{t})}^2w(t)\, dxdt\le CJ_u(1).
\]
This finishes the proof of \eqref{eq Jur step1}.

Now it suffices to control the second term on the right-hand side of \eqref{eq Jur step1}. We use the decomposition of $u_0$ as in \eqref{ext u0}, as well as \eqref{eq lambdaE} to write
\begin{multline}\label{3b32}
     m(B_{1/2})^{-1}\int_{B_{1/2}}\abs{\A-\A_0}^2\abs{\nabla u_0}^2dm\\
     \le \frac{2}{m(B_{1/2})}\int_{B_{1/2}}\abs{\A-\A_0}^2\abs{\nabla v_0}^2dm
     +\frac{2\lambda_1(u)^2}{m(B_{1/2})}\int_{B_{1/2}}\abs{\A-\A_0}^2\abs{\nabla \abs{t}}^2dm\\
    \le  \frac{2}{m(B_{1/2})}\int_{B_{1/2}}\abs{\A-\A_0}^2\abs{\nabla v_0}^2dm
    +2E_u(1)\gamma(0,1)^2.
\end{multline}
We claim that $m(B_{1/2})^{-1}\int_{B_{1/2}}\abs{\A-\A_0}^2\abs{\nabla v_0}^2dm$ can be estimated as in the $d=n-1$ case as long as one has the following reverse H\"older type estimates.

For some $p=p(d,n,\mu_0)>2$ sufficiently close to $2$,
 \begin{equation}\label{eq RH1}
     \br{\int_{B_{1/2}}\abs{\nabla v_0}^pdm}^{1/p}
  \lesssim \br{\int_{B_{1/2}}\abs{\nabla v_0}^2dm}^{1/2}
  +\br{\int_{B_{1/2}}\abs{\nabla v}^pdm}^{1/p},
 \end{equation}
 and 
 \begin{equation}\label{eq RH2}
     \br{\int_{B_{1/2}}\abs{\nabla v}^pdm}^{1/p}\lesssim \br{\int_{B_1}\abs{\nabla v}^2dm}^{1/2}+\abs{\lambda_1(u)}\br{\int_{B_1}\abs{\A-\A_0}^pdm}^{1/p},
 \end{equation}
where the implicit constants depend on $d$, $n$, $\mu_0$ and $p$.
We postpone the proof of these two inequalities to Section \ref{sec RH}.

Now fix any $K>0$. Assuming \eqref{eq RH1} and \eqref{eq RH2}, we can control the contribution from the set 
$$B_{\frac12}\setminus \set{X\in B_{\frac12}\setminus\Gamma: \abs{\nabla v_0(X)}^2\le KE_u(1)}$$
to the integral, much as in the case $d=n-1$, and finally obtain
\[
 \int_{B_{1/2}}\abs{\A-\A_0}^2\abs{\nabla v_0}^2dm
   \le  CK^{\frac{2-p}{2}}J_u(1)+C\br{K+K^{\frac{2-p}{2}}}\gamma(0,1)^2E_u(1).
\]
From this and \eqref{3b32}, the desired estimate \eqref{Jur} follows.
\end{proof}

Using Lemma \ref{lem Eu0 lwbd}, Lemma \ref{l319} and Lemma \ref{lem u=u^0 d<n-1}, one obtains the following analogue of Lemma \ref{lem Eu0 lwbd} for positive solutions of $\LL u=0$.

\begin{lem}\label{lem Eu lwbd}
Let $u\in W_r(B_5)$ be a positive solution of $\LL u=0$ in $B_5\setminus\Gamma$, with $Tu=0$ on $\Gamma\cap B_5$. 
Then for any $\delta>0$ and $0<r<1/2$,
\begin{equation} \label{3a30}
E_u(r)
\ge\br{ \frac{1-C' r^2}{C}  
-\frac{C''\br{\delta+C_\delta\gamma(0,1)^2}}{r^{d+1}}}
E_u(1),
\end{equation}
where $C$, $C'$, $C''$ are positive constants depending only on $d$, $n$ and $\mu_0$.
\end{lem}
As before, we will only find this useful when the parenthesis is under control. 
\ms

With Lemma \ref{lem Jur}, Lemma \ref{lem Eu lwbd}, and 
Lemma \ref{lem gamma} at hand, we are finally ready to prove the decay estimate for $\beta_u(x,r)$, the normalized non-affine energy of solutions of $\LL u=0$. 

Let $u$ be as in Lemma \ref{lem Eu lwbd}. 
We first choose a $\theta_0\in(0,1)$ so that $C\theta_0<\frac{1}{16C}$ in Lemma \ref{lem Jur}. 
By Lemma \ref{lem Jur}, this choice of $\theta_0$ gives an $r_0\in(0,1/4)$ such that \eqref{Jur} holds for any $r\le r_0$. Now we choose $r=\tau_0\le r_0$ so that $C'r^2<1/2$ in \eqref{3a30}. Then we require
\begin{equation}\label{eq gamma sm}
    \gamma(0,1)^2\le \epsilon_0,
\end{equation} and choose $\eps_0$ and $\delta>0$ sufficiently small (depending on $\tau_0$) so that 
\[C''\br{\delta+C_\delta\eps_0}\tau_0^{-d-1} < \frac1{4C}\] in \eqref{3a30}.
This way, \eqref{3a30} implies that
\begin{equation} \label{3a34}
E_u(r)
\ge  \frac{1}{4C}   E_u(1).
\end{equation}
 
We divide both sides of \eqref{Jur} by $E_u(r)$ and get that
\begin{equation} \label{3a36}
\beta_u(0,r) \leq C \br{\theta_0+K^{\frac{2-p}{2}}r^{-d-1}} \frac{J_u(1)}{E_u(r)} 
+\frac{C_K}{r^{d+1}}\gamma(0,1)^2\frac{E_u(1)}{E_u(r)}.
\end{equation}
Then we choose $K>0$ sufficiently small (depending on $\tau_0$) so that $CK^{\frac{2-p}{2}}\tau_0^{-d-1}<\frac{1}{16C}$. Now assuming \eqref{eq gamma sm}, our choice of $\theta$, $\eps_0$, $\delta$ and $K$ guarantees that we can apply \eqref{3a34} and deduce from \eqref{3a36} that 
\begin{equation}
    \beta_u(0,\tau_0) \leq\frac{1}{2}\beta_u(1) + C_{\tau_0}\gamma(0,1)^2.
\end{equation}

We recapitulate what we obtained in the next corollary.
Of course, by translation and dilation invariance, what was done on the unit ball $B_1$ can also 
be done for any other $B_R(x)$, $x\in\Gamma$, $R>0$. We use this opportunity to state the general case, which of course can easily be deduced from the case of $B_1$ by homogeneity.

\begin{cor}\label{cor beta}
There exist constants $\tau_0 \in (0,10^{-1})$ and $C > 0$
which depend only on $d$, $n$ and $\mu_0$, such that
if $u\in W_r(B_{5R}(x))$ is a positive solution of $\LL u=0$ 
in $B_{5R}(x)\setminus\Gamma$, with $Tu=0$ on $\Gamma\cap B_{5R}(x)$, 
then 
\begin{equation} \label{3a39}
\beta_u(x,\tau_0 R) \leq  \frac12 \beta_u(x,R) + C \gamma(x,R)^2.
\end{equation}
\end{cor}

\begin{proof}
The discussion above gives the result under the additional condition that 
$\gamma(x,R) \leq \varepsilon_0$. But we now have chosen $\tau_0$
and $\varepsilon_0$, and if $\gamma(x,R) > \varepsilon_0$, \eqref{3a39}
holds trivially (maybe with a larger constant), because $\beta_u(x,\tau_0 R) \leq 1$ by \eqref{eq beta<1}.
\end{proof}

Finally, Theorems \ref{mt} and \ref{mt2} 
can be deduced from the decay estimate \eqref{3a39} exactly as what was done in the $d=n-1$ case, as $\beta_u(x,r)$ is a function in $\Real^{d+1}_+$ and the goal is to prove a Carleson estimate in $\Real^{d+1}_+$. We refer readers to Section 4.2 in \cite{DLMgreen} for details.

\subsection{Proof of the Reverse H\"older inequalities}\label{sec RH}
\begin{proof}[Proof of \eqref{eq RH1}]
The idea of the proof is essentially from \cite{giaquinta1983multiple}, Chapter V. However, we need to treat the boundary estimates more carefully as this time the boundary is of mixed co-dimensions. 

 Recall that $\LL_0 v_0=0$ in $B_{1/2}\setminus\Gamma$, with $v_0-v\in W_0(B_{1/2}\setminus\Gamma)$. Since $v\in W(B_{1/2})$ with $Tv=0$ on $B_{1}\cap \Gamma$, $Tv_0=0$ on $B_{1/2}\cap \Gamma$. Let $R_0=10^{-2}n^{-1/2}$. Set
 \[
 Q_R(X):=\set{Y\in\Rn: \abs{Y_i-X_i}<R \ \text{ for } 
 i=1,2,\dots, n}, \quad R>0.
 \]
We claim that there exists 
$p=p(d,n,\mu_0)>2$ such that 
 \begin{multline}\label{eq RH1 QR0}
     \br{m(Q_{R_0/2}(X_0))^{-1}\int_{Q_{R_0/2}(X_0)\cap B_{1/2}}\abs{\nabla v_0}^pdm}^{1/p}\\
     \lesssim \br{m(Q_{R_0}(X_0))^{-1}\int_{Q_{R_0}(X_0)\cap B_{1/2}}\abs{\nabla v_0}^2dm}^{1/2}\\
     +\br{m(Q_{R_0}(X_0))^{-1}\int_{Q_{R_0}(X_0)\cap B_{1/2}}\abs{\nabla v}^pdm}^{1/p}
\end{multline}
for any $Q_{R_0}(X_0)\subset\Rn$ with 
 $Q_{R_0}(X_0)\cap B_{1/2}\neq\emptyset$. 
Notice that the first integral concerns the cube $Q_{R_0/2}(X_0)$, while the two other ones
are on the larger $Q_{R_0}(X_0)$; this will allow the localization argument below.
Once this is proved, one can obtain the desired estimate \eqref{eq RH1} 
by covering $B_{1/2}$ with finitely many cubes $Q_{R_0}(X_0)$.

Fix $Q_{R_0}(X_0)$ with $Q_{R_0/2}(X_0)\cap B_{1/2}\neq\emptyset$. 
Let $X\in Q_{R_0}(X_0)$ be given, and 
pick any radius
$R<\frac{1}{12}\dist(X,\bdy Q_{R_0}(X_0))$.
We need to introduce $R$ because we will apply a local result soon.

Let $q:=\frac{2n}{n+2}$. There are three possibilities: (1) $Q_{3R}(X)\subset B_{1/2}$, (2) $Q_{3R}(X)\cap B_{1/2}\neq\emptyset$ and $Q_{3R}(X)\cap \stcomp{ B_{1/2}}\neq\emptyset$, (3) $Q_{3R}(X)\subset \stcomp{B_{1/2}}$. The last situation is trivial.

If $Q_{3R}(X)\subset B_{1/2}$, then we can apply Lemma \ref{lem RH 2-2+} to get
\begin{equation}\label{eq RHcasein}
   m(Q_R(X))^{-1}\int_{Q_R(X)}\abs{\nabla v_0}^2dm\le \frac{C}{m(Q_{3R}(X))}\int_{Q_{3R}(X)}\abs{\nabla v_0}^q dm.
\end{equation}
 We will see later how to continue in this case, but let us first discuss (2).
If $Q_{3R}(X)\cap B_{1/2}\neq\emptyset$ and $Q_{3R}(X)\cap \stcomp{ B_{1/2}}\neq\emptyset$, choose $\eta\in C_0^\infty(Q_{3R}(X))$ with $\eta=1$ on $Q_R(X)$ and $\abs{\nabla\eta}\lesssim \frac{1}{R}$. Taking $(v-v_0)\eta^2\in W_0(B_{1/2}\setminus\Gamma)$ as a test function in $\LL_0v_0=0$, and using the ellipticity conditions on $\A_0$, and then the 
Cauchy-Schwarz inequality, one can get the estimate
\begin{equation}\label{eq RHv0}
    \int_{Q_R(X)\cap B_{1/2}}\abs{\nabla v_0}^2dm\le C\int_{Q_{3R}(X)\cap B_{1/2}}\abs{\nabla v}^2dm +\frac{C}{R^2}\int_{Q_{3R}(X)\cap B_{1/2}}\abs{v_0-v}^2dm.
\end{equation}
We want to control $\int_{Q_{3R}(X)\cap B_{1/2}}\abs{v_0-v}^2dm$ using the 
Poincar\'e inequality. Extend $v_0-v$ by zero outside $B_{1/2}$ and denote by $h$ the extended function. We need to discuss two cases. 

\textbf{Case 1:} $Q_{4R}(X)\cap\Gamma=\emptyset$. Then $\delta(X)\ge 4R$, where $\delta(X)=\dist(X,\Gamma)$ as usual. Since for any $Z\in Q_{3R}(X)$, $\delta(X)-3R\le\delta(Z)\le\delta(X)+3R$, we have 
\(
\frac{1}{4}\le \frac{\delta(Z)}{\delta(X)}\le\frac{7}{4}
\).
This implies that
\[
C_{n,d}w(X)\le w(Z)\le C_{n,d}w(X) \qquad\text{for }Z\in Q_{3R}(X),
\]
and thus
\[
    \int_{Q_{3R}(X)}\abs{h(Z)}^2w(Z)dZ\le C_{n,d}w(X)\int_{Q_{3R}(X)}\abs{h(Z)}^2dZ.
\]
Since $\bdy B_{1/2}$ is smooth, $Q_{3R}(X)\cap \stcomp{B_{1/2}}\neq\emptyset$ implies that $\abs{Q_{7R/2}(X)\setminus B_{1/2}}\ge\gamma\abs{Q_{7R/2}(X)}$ for some $\gamma>0$. 
Recalling that $h=0$ in $\stcomp{B_{1/2}}$, we can apply the Sobolev inequality to get
\begin{multline*}
    \int_{Q_{3R}(X)}\abs{h(Z)}^2w(Z)dZ\le Cw(X)\Big(\int_{Q_{7R/2}(X)}\abs{\nabla h}^qdZ\Big)^{\frac{2}{q}}\\
    \le Cw(X)^{-\frac{2}{n}}\Big(\int_{Q_{7R/2}(X)}\abs{\nabla h}^qw(Z)dZ\Big)^{\frac{2}{q}}.
\end{multline*}
Notice that by \eqref{mBr far}, $m(Q_{3R}(X))\approx m(Q_{7R/2}(X))\approx R^n w(X)$. Hence,
\begin{equation*}
    \frac{1}{m(Q_{3R}(X))}\int_{Q_{3R}(X)}h^2dm\le C R^2\Big(m(Q_{7R/2}(X))^{-1}\int_{Q_{7R/2}(X)}\abs{\nabla h}^qw(Z)dZ\Big)^{\frac{2}{q}}.
\end{equation*}
\textbf{Case 2:} $Q_{4R}(X)\cap\Gamma\neq\emptyset$. Then there is $x_0\in\Gamma$ so that $Q_{3R}(X)\subset Q_{7R}(x_0)\subset Q_{11R}(X)$. Enlarging $Q_{3R}(X)$ and applying \eqref{eq Poincare u=0 2-2+}, one has 
\begin{multline*}\label{eq v0-v}
     \frac{1}{m(Q_{3R}(X))}\int_{Q_{3R}(X)}h^2dm\le\frac{C}{m(Q_{7R}(x_0))}\int_{Q_{7R}(x_0)}h^2dm\\
     \le CR^2\Big(m(Q_{7R}(x_0))^{-1}\int_{Q_{7R}(x_0)}\abs{\nabla h}^qdm\Big)^{2/q}\\
     \le CR^2\Big(m(Q_{11R}(X))^{-1}\int_{Q_{11R}(X)}\abs{\nabla h}^qdm\Big)^{2/q}.
\end{multline*}

To summarize, in both Case 1 and Case 2, 
we have
\begin{multline}
    \frac{1}{m(Q_{3R}(X))}\int_{Q_{3R}(X)\cap B_{1/2}}\abs{v_0-v}^2dm\\
    \le C R^2\Big(\frac{1}{m(Q_{11R}(X))}\int_{Q_{11R}(X)\cap B_{1/2}}\abs{\nabla(v_0-v)}^qdm\Big)^{\frac{2}{q}}. 
\end{multline}
Notice that we have chosen $R<\frac{1}{12}\dist(X,\bdy Q_{R_0}(X_0))$ to make sure $Q_{11R}(X)\subset Q_{R_0}(X_0)$.
Set 
\[
g(X)=\begin{cases}
\abs{\nabla v_0(X)}^q \qquad\text{for }X\in Q_{R_0}(X_0)\cap B_{1/2},\\
0 \qquad \text{otherwise},
\end{cases}
\]
\[
f(X)=\begin{cases}
\abs{\nabla v(X)}^q \qquad\text{for }X\in Q_{R_0}(X_0)\cap B_{1/2},\\
0 \qquad \text{otherwise}.
\end{cases}
\]
By \eqref{eq v0-v}, \eqref{eq RHv0} and \eqref{eq RHcasein}, we obtain
\begin{multline*}
    \frac{1}{m(Q_R)}\int_{Q_R}g^rdm\le\frac{C}{m(Q_{3R})}\int_{Q_{3R}}f^rdm\\
    +C\Big(\frac{1}{m(Q_{11R})}\int_{Q_{11R}}gdm\Big)^r+C\Big(\frac{1}{m(Q_{11R})}\int_{Q_{11R}}fdm\Big)^r\\
    \le \frac{C}{m(Q_{11R})}\int_{Q_{3R}}f^rdm
    +C\Big(\frac{1}{m(Q_{11R})}\int_{Q_{11R}}gdm\Big)^r,
\end{multline*}
where $r=\frac{n+2}{n}$. As we noted in the proof of Lemma \ref{lem RH}, 
we can still apply Proposition 1.1 in Chapter V of \cite{giaquinta1983multiple}
when the Lebesgue measure is replaced with the doubling measure $m$. Then \eqref{eq RH1 QR0} follows.
\end{proof}

\begin{proof}[Proof of \eqref{eq RH2}] The proof is similar to that in the $d=n-1$ case. We present the proof for the sake of completeness. 

Set $R_0=10^{-2}n^{-1/2}$. 
For any $X_0=(x_0,t_0)\in B_{1/2}\setminus\Gamma$ and 
$0<R\le R_0$, choose $\eta\in C_0^\infty(Q_{R}(X_0))$, with $\eta\equiv 1$ in $Q_{2R/3}(X_0)$, $\abs{\nabla \eta}\lesssim 1/R$. Here, 
\(
 Q_R(X)=\set{Y\in\Rn: \abs{Y_i-X_i}<R \quad i=1,2,\dots, n}
 \) as before.
We shall write $Q_R$ for $Q_R(X_0)$ when this does not cause a confusion.
Since $u\in W(B_{1/2})$ verifies $\LL u=0$ in $B_1\setminus\Gamma$, we can take any $\vp\in W_0(B_1\setminus\Gamma)$ as test function 
(see \eqref{eq testfucW0}). 
Moreover, recall that 
$v(x,t)=u(x,t)-\lambda\abs{t}$, with $\lambda=\lambda_1(u)$, 
and that $\LL_0 \abs{t}=0$. 
Therefore, for any $\vp\in W_0(B_1\setminus\Gamma)$,
\begin{multline}\label{RH2 eq1}
    0=\int_{B_1}\A\nabla u\cdot\nabla \vp dm=\int_{B_1}\A\nabla v\cdot\nabla \vp dm + \int_{B_1}\A\nabla (\lambda \abs{t})\cdot\nabla\vp dm\\
    =\int_{B_1}\A\nabla v\cdot\nabla\vp dm 
    + \int_{B_1}(\A-\A_0)\nabla (\lambda \abs{t})\cdot\nabla \vp dm . 
\end{multline}

When $\abs{t_0}\le R$, we choose $\vp(X)=v(X)\eta^2(X)$; when instead
$\abs{t_0}>R$, we take $\vp(X)=(v(X)-v_{Q_R})\,\eta^2(X)$, with
$v_{Q_R}=m(Q_R)^{-1}\int_{Q_R}vdm$.
One can check that in both cases $\vp\in W_0(B_1\setminus\Gamma)$.  As in the proof of (3.34) in \cite{DLMgreen}, we plug $\vp$ into \eqref{RH2 eq1}, compute the derivatives, estimate some terms brutally,
and finally use Cauchy-Schwarz inequality, and get the following estimates. \\
\textbf{Case 1:} $\abs{t_0}\le R$.
In this case, we obtain 
\begin{equation*}
     \int_{Q_{2R/3}}\abs{\nabla v}^2dm
    \le \frac{C_{\mu_0}}{R^2}\int_{Q_R}v^2dm
        + C_{\mu_0}\abs{\lambda}^2\int_{Q_R}\abs{\A-\A_0}^2dm.
\end{equation*}
There is $x_0\in\Gamma$ such that $Q_R\subset Q_{2R}(x_0)\subset Q_{3R}$. Since $Tv=0$ on $\Gamma\cap B_1$, we can enlarge $Q_R$ and apply \eqref{eq Poincare u=0 2-2+} to control $\int_{Q_R}v^2dm$
and deduce from the above that
\begin{multline}\label{RH2 eq2}
    \frac{1}{m(Q_{2R/3})}\int_{Q_{2R/3}}\abs{\nabla v}^2dm\\
    \le C\br{\frac{1}{m(Q_{3R})}\int_{Q_{3R}}\abs{\nabla v}^{\frac{2n}{n+2}}dm}^{\frac{n+2}{n}}+ \frac{C\abs{\lambda}^2}{m(Q_R)}\int_{Q_R}\abs{\A-\A_0}^2dm.
\end{multline}
\textbf{Case 2:} $\abs{t_0}> R$.
The same computation as in Case 1 gives
\[    \int_{Q_{2R/3}}\abs{\nabla v}^2dm
    \le\frac{C}{R^2}\int_{Q_R}\abs{v-v_{Q_R}}^2dm+ C\abs{\lambda}^2\int_{Q_R}\abs{\A-\A_0}^2dm.
\]
Then by Lemma \ref{lem Poincare}, \eqref{RH2 eq2} holds. 

Now it follows from \cite{giaquinta1983multiple} V, Proposition 1.1 that
\begin{multline*}
   \frac{1}{m(Q_{R_0/2})}\int_{Q_{R_0/2}}\abs{\nabla v}^pdm\\
\le C\br{\frac{1}{m(Q_{R_0})}\int_{Q_{R_0}}\abs{\nabla v}^2dm}^{\frac{p}{2}}
+ \frac{C\abs{\lambda}^p}{m(Q_{R_0})}\int_{Q_{R_0}}\abs{\A-\A_0}^pdm, 
\end{multline*}
for some $p=p(d,n,\mu_0)>2$, which implies the desired reverse 
H\"older type estimate since $B_{1/2}$ can be covered by finitely many $Q_{R_0/2}$.
\end{proof}

\bibliographystyle{plain}
\bibliography{reference}

\end{document}